\documentclass{article} 
\usepackage[dvipdfm]{graphicx}
\usepackage{color}
\usepackage{amsmath, amssymb, amsthm} 
\usepackage{verbatim}                  
\usepackage[alphabetic]{amsrefs}       

\theoremstyle{plain} 
\newtheorem{theorem}{Theorem}
\newtheorem{lemma}[theorem]{Lemma}
\newtheorem{corollary}[theorem]{Corollary}
\newtheorem{proposition}[theorem]{Proposition}

\newtheorem*{theorem0}{Theorem}

\theoremstyle{definition}

\newtheorem*{acknowledgements}{Acknowledgements}

\newcommand{\tr}{\operatorname{tr}}
\newcommand{\bl}{\operatorname{bl}}
\title{A percolation on directed graphs}
\author{Mamoru Tanaka\footnote{This work was supported by World Premier International Research Center Initiative (WPI), MEXT, Japan}}
\date{} 

\begin{document}

\maketitle

\begin{abstract}
Suppose each site independently and randomly chooses some sites around it, and it is weakly (strongly) connected with them (if there choose each other). 
What is the probability that the weak (strong) connected cluster is infinite?  
We investigate a percolation model for this problem, which is a generalization of site percolation. 
We give a relation between the probability of the number of chosen sites around a site and the size of clusters. 
We also see the expected number of infinite clusters, and the exponential tail decay of the radius and the size of a cluster. 



\end{abstract}


\section{Introduction} 
Suppose a case are filled by molecules and each molecule has some arms for grabbing other molecules. 
When the molecules make a very large molecule cluster? 
We investigate a percolation model for such a problem in mind. 

Percolation is a model for representing randomly connected components consisting of sites on graphs (cf.\ \cite{MR1707339}). 
For example, in bond percolation in \cite{MR0091567}, we choose open edges independently and randomly, and connect sites at the end of them.
In site percolation, we choose open sites independently and randomly, and connect each open site and open sites around it.  
For directed graphs, there also exist some percolation models: 
For oriented graphs, whose edges are determined their direction, in oriented percolation in \cite{MR1707339}, we choose open edges independently and randomly, and investigate the existence of infinite directed open paths. 
For directed graphs whose edges are not determined their direction, in random-oriented percolation in \cite{MR1840826}, we choose the directions of every edges independently and randomly, and investigate the existence of infinite directed paths. 
For the square lattice, random-oriented percolation in \cite{MR1824275}, we choose "rightwards", "leftwards", "rightwards and leftwards" or "absent" at each horizontal edge, and similarly "upwards" and "downwards" at each vertical edge independently and randomly, and investigate the existence of infinite directed paths. 

We consider the following percolation model on directed graphs: 
Let $G=(V,E,o)$ be an infinite, simple directed graph with the origin $o$, that is, 
the vertex set $V$ is a countably infinite set, the edge set $E $ is a subset of $V\times V \backslash \{(x,x) \mid x \in V\}$, and $o$ is a fixed vertex. 
Vertices are also called {\it sites}.
We assume $(x,y)\in E$ if and only if $(y,x)\in E$. 
We also assume $G$ is {\it $d$-regular}, that is, $|\{ y \in V \mid (x,y)\in E \}|=d$ for every $x \in V$.  
We write $x \sim y$ for $ x,y \in V$ when $(x,y)\in E$. 
We suppose that $G$ is {\it connected}, that is, for any $x,y \in V$ there are $z_1, z_2, \dots, z_{k-1} \in V$ such that $x \sim z_1 \sim z_2 \sim \dots \sim z_{k-1} \sim y$. 
The sequence $(x,z_1,z_2,\dots,z_{k-1},y)$ is called a (non-directed) {\it path} from $x$ to $y$ in $G$ with the length $k$. 
The graph distance $\delta(x,y)$ between vertices $x$ and $y$ in $G$ is the length of the shortest paths connecting them. 
Denote $B(n):=\{x \in V \mid \delta(0,x) \le n\} $ and $\partial B(n):=\{x \in V \mid \delta(o,x) = n\}$. 
A {\it graph isomorphism} $f:V \to V'$ between directed graphs $G=(V,E)$ and $G'=(V',E')$ is a bijective map such that $(x,y) \in E$ if and only if $(f(x), f(y)) \in E'$.  
In particular, if $G=G'$, then it is called a {\it graph automorphism}. 
We assume $G$ is {\it vertex-transitive}, that is, for any pair of vertices $x,y \in V$ there exists a graph automorphism $f$ satisfying $f(x)=y$. 
Since $G$ is infinite and connected, we have $d\ge 2$ and if $d=2$ then $G$ is uniquely determined up to isomorphism.   

A {\it subgraph} $G'=(V',E')$ of $G$ is a directed graph with $V'\subset V$ and $E'\subset E$ such that if $(x,y) \in E'$ then $x,y \in V'$, which may not satisfy $(x,y) \in E'$ if and only if $(y,x)\in E'$.   
We denote $x \leftrightharpoons y$ in $G'$ for $x,y \in V'$ with $x \sim y$ if $(x,y) \in E'$ or $(y,x)\in E'$. 
For  $x,y \in V'$, if $x = y$ or if there are $z_1, z_2, \dots, z_k\in V$ such that $x \leftrightharpoons z_1 \leftrightharpoons  z_2 \leftrightharpoons  \dots \leftrightharpoons  z_k \leftrightharpoons y$, then we say $x$ and $y$ are {\it weakly connected} and write as $x \leftrightharpoons y$ in $G'$. 
Similarly, we denote $x \leftrightarrow y$ in $G'$ for $x,y \in V'$ with $x \sim y$ if $(x,y)\in E'$ and $(y,x)\in E'$. 
For  $x,y\in V'$, if $x = y$ or if there are $z_1, z_2, \dots, z_k \in V$ such that $x \leftrightarrow z_1 \leftrightarrow z_2 \leftrightarrow \dots \leftrightarrow z_k \leftrightarrow y$, then we say $x$ and $y$ are {\it strongly connected} and write as $x \leftrightarrow y$ in $G'$.  
We denote the cardinality of $V'$ by $|V'|$ or $|G'|$. 

Let $L_x:=\{y \in V \mid (x,y) \in E \}$ for every vertex $x\in V$. 
We give an arbitrary order of the elements in $L_x$, fix it, and denote it as $\{x(1),x(2), \dots,x(d) \}$. 
The set $2^{L_x}$ is the power set of $L_x$. 	
As sample space we take $\Omega = \prod_{x \in V}2^{L_x}$, points of which are represented as $\omega = \{\omega (x) \subset L_x : x \in V\}$ and called {\it configurations}. 
Each $\omega(x) $ is called the {\it state} at $x$. 
We take $\mathcal{F}$ to be the $\sigma $-field of subsets of $\Omega$ generated by the finite dimensional cylinders. 
An element $A \in \mathcal{F}$ is called an {\it event}. 
Let $p_0,p_1,\dots ,p_d$ be non-negative numbers satisfying $ p_0+p_1+\cdots +p_d= 1$, and set ${\bf p}=(p_0,p_1,\dots ,p_d)$. 
We take a product measure on $(\Omega , \mathcal{F})$ defined as  
$P_{\bf p} = \prod_{x \in V} \mu_x $
where $\mu_x$ is given by 
\begin{eqnarray*} 
\mu_x (\omega(x) = \emptyset) = p_0 , \ \ 
\mu_x (\omega(x) = \{x(i_1),\dots, x(i_k) \}) = \frac{p_k}{\binom{d}{k}} = \frac{k!(d-k)!}{d!}p_k
\end{eqnarray*}
for each $\{x(i_1),\dots, x(i_k) \} \subset L_x$. 
This means each vertex $x$ chooses $k$ vertices around it with probability $p_k/\binom{d}{k}$ for $k=0,1,\dots,d$. 
Each configuration $\omega $ can be regarded as the subgraph $G_\omega$ of $G$ with the vertex set $V$ and the edge set 
\begin{eqnarray*}
E_{\omega} := \{(x,y) \in V \times V \mid x \in V, y \in \omega(x) \}.
\end{eqnarray*}
The {\it weak cluster} of $G_\omega$ containing $x \in V$ is the subgraph $C(x)$ such that the vertex set is 
$\{y \in V \mid x \leftrightharpoons y \text{ in } G_\omega \}$, 
and the edge set is 
$\{(y,z) \in E_\omega \mid x \leftrightharpoons y, x \leftrightharpoons z \text{ in } G_\omega \}$. 
The {\it strong cluster} of $G_\omega$ containing $x \in V$ is the subgraph $\tilde{C}(x)$ such that the vertex set is 
$\{y \in V \mid x \leftrightarrow y \text{ in } G_\omega \}$, 
and the edge set is 
$\{(y,z) \in E_\omega \mid x \leftrightarrow y, x \leftrightarrow z \text{ in } G_\omega \}$. 
When $x=o$, we abbreviate $C=C(o)$ and $\tilde{C}=\tilde{C}(o)$. 
If $p_0=1-p$ and $p_d=p$ for $0\le p\le 1$, then this model with the strong connection is equivalent to the site percolation.  
Hence this model is a generalization of the site percolation. 

We define $\omega_1 \le \omega_2$ if $\omega_1(x) \subset \omega_2(x) $ for all $x \in V$.
The event $A \in \mathcal{F}$ is called {\it increasing} if $I_A(\omega_1) \le I_A(\omega_2)$ whenever $\omega_1 \le \omega_2 $, where $I_A$ is the indicator function of $A$. 
A random variable $N$ on $(\Omega, \mathcal{F})$ is also called  {\it increasing} if $N(\omega_1) \le N(\omega_2)$ whenever $\omega_1 \le \omega_2 $. 
Let $(X(x):x \in V)$ be independent random variables uniformly distributed on $[0,1)$.  
Order the set $\{(i_1,\dots,i_d) \in \{1,\dots ,d \}^d \mid  i_j \not= i_k \text{ if } j \not= k \}$ with respect to the lexicographic order, and denote by $a_j$ the $j$-th element of this set for $j=1,2,\dots,d!$. 
We define 
\begin{eqnarray*}
\eta_{\bf p}(x):=
\begin{cases}
\emptyset  & \text{ if } X(x) \in \left[ \frac{j-1}{d!}, \frac{j-1+p_0}{d!}\right) \text{ for some } j\\
\{x(i_1), \dots, x(i_k) \} & \text{ if } X(x) \in \left[ \frac{j-1+p_0+\dots +p_{k-1}}{d!}, \frac{j-1+p_0+\dots +p_k}{d!}\right) \\ 
&\ \ \  \text{ for } j \text { with } a_j=(i_1,\dots,i_d). 
\end{cases}
\end{eqnarray*}
Then we have $P(\eta_{\bf p}(x) = S)=P_{\bf p}(\omega(x) = S)$ for any $S \subset L_x$. 
For ${\bf p} = (p_0,p_1,\dots,p_d)$ and ${\bf p}'=(p_0',p_1',\dots,p_d')$, we define ${\bf p} \le {\bf p}'$ if $p_i+\dots+p_d \le p_i'+\dots+p_d'$ for all $i=0,1,\dots,d$. 
If ${\bf p} \le {\bf p'}$, then $\eta_{\bf p}(x) \subset \eta_{\bf p'}(x)$ for all $x \in V$. 
This gives that $N(\eta_{\bf p}) \le N(\eta_{\bf p'})$ for any increasing random variable $N$. 
Hence 
\begin{eqnarray*}
\mathbb{E}_{\bf p}(N) = \mathbb{E}(N(\eta_{\bf p})) \le \mathbb{E}(N(\eta_{\bf p'})) = \mathbb{E}_{\bf p'}(N), 
\end{eqnarray*}
for any increasing random variable $N$, and 
\begin{eqnarray*}
P_{\bf p}(A) = \mathbb{E}_{\bf p}(I_A)  \le  \mathbb{E}_{\bf p'}(I_A) =  P_{\bf p'}(A) 
\end{eqnarray*}
for any increasing event $A$. 

We would like to know the behavior of $P(|C|=k)$ and $P(|\tilde{C}|=k)$ for $k=1,2,3,\dots $, and 
\begin{eqnarray*}
\theta (G,{\bf p}):= P(|C|=\infty)  \ \ \ \text{ and } \ \ \  \tilde{\theta} (G,{\bf p}):= P(|\tilde{C}|=\infty), 
\end{eqnarray*} 
and the expectations 
\begin{eqnarray*}
\chi(G,{\bf p}):=\mathbb{E}_{\bf p}(|C|) \ \ \ \text{ and } \ \ \  \tilde{\chi}(G,{\bf p}):=\mathbb{E}_{\bf p}(|\tilde{C}|) .  
\end{eqnarray*} 
We can easily see that $\theta (G,{\bf p})=0$ if $p_0=1$, $\tilde{\theta} (G,{\bf p})=0$ if $p_0+p_1=1$, and 
$\theta (G,{\bf p})=\tilde{\theta} (G,{\bf p})=1$ if $p_d=1$.  
Since the events $\{|C|=\infty\}$ and $\{|\tilde{C}|=\infty\}$ are increasing, 
the probabilities $\theta (G,{\bf p})$ and $\tilde{\theta} (G,{\bf p})$ are nondecreasing in ${\bf p}$. 

If $G$ is a square lattice and $p_0=\frac{1}{3}, p_1=\frac{1}{3}, p_2=\frac{1}{3}$, then we can see a configuration $\omega$ as the below image: 
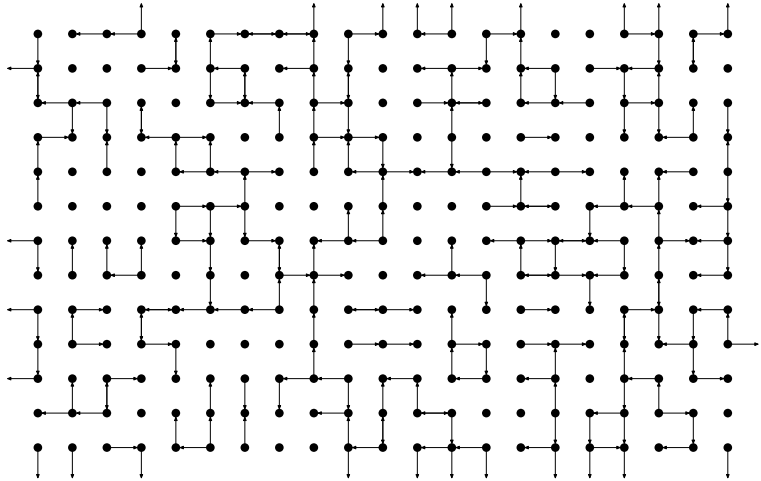
\begin{center}
\includegraphics[width=10cm]{DVP-p.126}  
\end{center}

In Section \ref{sec:2-reg}, we calculate $\chi(T_2,{\bf p}) $ and $\tilde{\chi}(T_2,{\bf p})$ for the connected infinite directed 2-regular graph $T_2$,  
which is the most simple infinite directed regular graph. 

In Section \ref{sec:tree}, using the theory of branching processes, we see that the $d$-regular directed tree $T_d$ ($d \ge 3$), that is, the connected $d$-regular directed graph without cycles, satisfies 
$\theta(T_d,{\bf p})=0$ if $p_0+p_1=1$ and $\theta(T_d,{\bf p})>0$ if $p_2+\cdots + p_d=1$.  
Moreover, we investigate the critical probability when ${\bf p}$ satisfies $p_1+p_2=1$. 
We also see that 
$\tilde{\theta}(T_d,{\bf p}) = 0$ if $p_0+\cdots+p_k=1$, and $\tilde{\theta}(T_d,{\bf p}) > 0$ if $p_{k+1}+\cdots+p_d=1$ for $k(k-1) \le d <  (k+1)k$. 
Moreover, we investigate the critical probability when ${\bf p}$ satisfies $p_k+p_{k+1}=1$ for $k$ with $k(k-1) < d <  (k+1)k$. 

In section \ref{sec:FiniteCluster}, we give sufficient conditions of $\theta(G,{\bf p}) = 0$ and $\tilde{\theta}(G,{\bf p}) = 0$. 
\begin{theorem0}
If $d\ge 3$, then $\theta(G,{\bf p})=0$ if $p_0+p_1$ is sufficiently near to $1$, and  
 $\tilde{\theta}(G,{\bf p})=0$ if $p_0 + \cdots +p_k$ is sufficiently near to $1$ for $k$ with $k(k-1) < d$. 
\end{theorem0}
As we will see in section \ref{sec:InfiniteCluster}, for the triangular lattice ($d=6$), $\tilde{\theta}(G,{\bf p})>0$ if $p_3$ is sufficiently near to $1$.
Hence the conclusion of the theorem does not hold for $d=k(k-1)$. 

In section \ref{sec:InfiniteCluster}, we give sufficient conditions of $\theta(G,{\bf p}) > 0$ and $\tilde{\theta}(G,{\bf p}) > 0$ 
for some vertex-transitive directed $d$-regular planar graph.
A {\it self-avoiding path} with length $n$ in $G$ is a path $(x_0, x_1, \dots ,x_n)$ such that $x_i\not =x_j$ for all $i\not=j$. 
Let $\sigma_G (n)$ be the number of the self-avoiding paths in $G$ having length $n$ and beginning at the origin. 
Since $G$ is a connected infinite $d$-regular graph, $1 \le \sigma_G(n) \le (d-1)^n $. 
The {\it connective constant} of $G$ is given by 
$$\lambda(G) := \lim_{n \to \infty } \sigma_G(n)^{\frac{1}{n}}.  $$
This limit exists and satisfies $1 \le \lambda(G) \le d-1$ (\cite{MR1197356}). 
Planar graphs $G$ have their dual graphs $G^*$ (see section \ref{sec:InfiniteCluster} for the definition of a directed planar graph and its dual graph). 
Let $N(k)$ be the length of the shortest paths in $G^*$ such that $B(k)$ in $G$ is enclosed.  
\begin{theorem0} 
Assume $G$ is isomorphic to a directed planar graph with $N(k) \to \infty$ as $k \to \infty$. 
\begin{enumerate}
\item If $d =k $ or $\lambda(G^*)  < \left(\frac{d}{d-k}\right)^2$, 
and $p_k+\cdots + p_d$ is sufficiently near to $1$, then $\theta(G,{\bf p})>0$. 
\item If $d =k $ or 
$\lambda(G^*)  < \frac{d}{d-k} $, 
and $p_k+\cdots + p_d$ is sufficiently near to $1$, then $\tilde{\theta}(G,{\bf p})>0$. 
\end{enumerate}  
\end{theorem0}
Note that since $\lambda(G^*) \le d^*-1$, we can give the positivity of  $\theta(G,{\bf p})$ and  $\tilde{\theta}(G,{\bf p})$ only checking a relation between $d$, $k$ and the maximum degree $d^*$ of the dual graph. 
For example, since the dual graph of the square lattice is the square lattice, and the dual graph of the hexagonal lattice is the triangular lattice, the dual graph of the triangular lattice is the hexagonal lattice, we can see that these graphs satisfy the assumption of the above theorem with $k=2$ for (i).  
We can also check that  the square lattice and the hexagonal lattice satisfy the assumptions of the above theorem with $k=3$ for (ii). 
On the other hand, for the triangular lattice, using the fact that the connective constant of the hexagonal lattice is $\sqrt{2+\sqrt{2}}$ (\cite{MR2912714}), we can check the assumptions of the above theorem with $k=3$ for (ii). 
In addition to these, we can use this theorem for other directed planer graphs. 

In section \ref{sec:Ineq}, we see the FKG inequality is not valid for our model. 
We remember Reimer's inequality, which is valid for any events. 
We also see Russo's formula for our model. 
Our model has $d+1$-dimensional parameter space of ${\bf p}$. 
Hence Russo's formula represents the directional derivative of $P_{\bf p}(A)$ on this space 
for an event $A$ which depends only on the states of the finite vertices. 

In section \ref{sec:Number}, we note that the number of weak clusters and strong clusters of per vertex are $\mathbb{E}_{\bf p} (|C|^{-1})$ and $\mathbb{E}_{\bf p} (|\tilde{C}|^{-1})$ for amenable graphs as the bond percolation in \cite{MR1707339}. 
We see that the number of weak infinite clusters and strong infinite causers are $0$, $1$ or $\infty$.  
In particular, if $G$ is amenable, then the number of weak infinite clusters is $0$ or $1$ when $p_0+p_1>0$, $p_2>0$, and 
the number of strong infinite clusters is $0$ or $1$ when $p_0+p_1>0$, $p_2>0$, $p_3+\cdots + p_d>0$.

In section \ref{sec:expdecay}, we will see that if $\chi(G,{\bf p}) < \infty$, then $P_{\bf p}(o \leftrightharpoons \partial B(n))$ and $P_{\bf p}(|C|\ge n ) $ are exponentially decay in $n$. 
Moreover, for the hypercube lattice $\mathbb{L}^D$ there exists $\lim_{n \to \infty} \{-\log P_{\bf p} (|C|=n)/n \} $ if $p_2 > 0$ (or $p_i>0$ for all $i \ge j$ with $p_j>0$, $j\ge 1$) and $D \ge 2$. 
Similarly if $\tilde{\chi}(G,{\bf p}) < \infty$, then $P_{\bf p}(o \leftrightarrow \partial B(n))$ and $P_{\bf p}(|\tilde{C}|\ge n )$ are exponentially decay in $n$,  
and for $\mathbb{L}^D$ there exists $\lim_{n \to \infty} \{-\log P_{\bf p} (|\tilde{C}|=n)/n \} $ if $D \ge 2$ and $p_i>0$ for all $i \ge j$ with $p_j>0$, $j\ge 1$.

\begin{acknowledgements}
The motivation for this research is come from a joint work with Akihiko Hirata of Tohoku University. 
We would like to thank Demeter Kiss for many valuable comments. 
We would also like to thank Masato Takei for comments on related results.  
\end{acknowledgements}

\section{2-regular infinite graph}\label{sec:2-reg}

The connected infinite directed 2-regular graph $T_2$ is the directed graph isomorphic to the directed graph with $V=\mathbb{Z}$ and $(x,y) \in E$ if and only if $|x-y|=1$. 
This is the most simple infinite directed regular graph. 

\begin{theorem}\label{thm:2-regular}
If $p_2 \not =1$, then we obtain 
\begin{eqnarray*}
\chi(T_2,{\bf p}) = \frac{8(1+p_0) }{(2p_0+p_1)^2}- 3  \ \ \ \text{ and }\ \ \ 
\tilde{\chi}(T_2,{\bf p}) = \frac{(2p_2+p_1)^2 }{2(1-p_2)} + 1.  
\end{eqnarray*}
\end{theorem}

\begin{proof}
Since $2^{L_n}=\{\emptyset, \{n-1\}, \{n+1\} , \{n-1,n+1\} \}$ for each $n\in \mathbb{N}$,  
we have 
\begin{eqnarray*}
&& P_{\bf p}(0 \leftrightharpoons n) \\ 
&=& P_{\bf p}(0 \leftrightharpoons n, \omega(n)=\emptyset ) 
+ P_{\bf p}(0 \leftrightharpoons n, \omega(n)=\{n+1\} ) \\ 
&& + P_{\bf p}(0 \leftrightharpoons n, \omega(n)=\{n-1\} ) 
+ P_{\bf p}(0 \leftrightharpoons n, \omega(n)=\{n-1,n+1\} ) \\
&=& P_{\bf p}(0 \leftrightharpoons n-1, \omega(n-1)=\{n\})p_0 + P_{\bf p}(0 \leftrightharpoons n-1, \omega(n-1)=\{n-2, n\})p_0 \\  
&& + P_{\bf p}(0 \leftrightharpoons n-1, \omega(n-1)=\{n\})\frac{p_1}{2} + P_{\bf p}(0 \leftrightharpoons n-1, \omega(n-1)=\{n-2, n\})\frac{p_1}{2} \\ 
&& + P_{\bf p}(0 \leftrightharpoons n-1)\frac{p_1}{2}  
+ P_{\bf p}(0 \leftrightharpoons n-1)p_2 \\ 
&=& P_{\bf p}(0 \leftrightharpoons n-1, \omega(n-1)=\{n\}) \left( p_0 + \frac{p_1}{2}\right) 
+ P_{\bf p}(0 \leftrightharpoons n-2)p_2\left( p_0 + \frac{p_1}{2}\right) \\  
&& + P_{\bf p}(0 \leftrightharpoons n-1)\left( \frac{p_1}{2} + p_2\right) . 
\end{eqnarray*}
for $n \ge 2$. 
Since  
\begin{eqnarray*}
&& P_{\bf p}(0 \leftrightharpoons n-1, \omega(n-1)=\{n\}) \\ 
&=&  P_{\bf p}(0 \leftrightharpoons n-2, \omega(n-2)=\{n-1\}, \omega(n-1)=\{n\}) \\ 
&& + P_{\bf p}(0 \leftrightharpoons n-2, \omega(n-2)=\{n-3, n-1\}, \omega(n-1)=\{n\}) \\ 
&=&  P_{\bf p}(0 \leftrightharpoons n-2, \omega(n-2)=\{n-1\})\frac{p_1}{2} 
\ +\  P_{\bf p}(0 \leftrightharpoons n-3)p_2\frac{p_1}{2} ,
\end{eqnarray*}
for $n \ge 3$, we get 
\begin{eqnarray*}
&& P_{\bf p}(0 \leftrightharpoons n)  - P_{\bf p}(0 \leftrightharpoons n-1)\left( \frac{p_1}{2} + p_2\right) 
- P_{\bf p}(0 \leftrightharpoons n-2)p_2\left( p_0 + \frac{p_1}{2}\right)\\ 
&=& \left( P_{\bf p}(0 \leftrightharpoons n-1)  - P_{\bf p}(0 \leftrightharpoons n-2)\left( \frac{p_1}{2} + p_2\right) 
- P_{\bf p}(0 \leftrightharpoons n-3)p_2\left( p_0 + \frac{p_1}{2}\right)\right)\frac{p_1}{2}\\ 
&& + P_{\bf p}(0 \leftrightharpoons n-3)p_2\frac{p_1}{2} \left( p_0 + \frac{p_1}{2}\right) . 
\end{eqnarray*}
These give the recurrence formula 
\begin{eqnarray*}
P_{\bf p}(0 \leftrightharpoons n) - (p_1+p_2) P_{\bf p}(0 \leftrightharpoons n-1)  
+ \left( \left( \frac{p_1}{2}\right) ^2 -p_0p_2\right) P_{\bf p}(0 \leftrightharpoons n-2) = 0. 
\end{eqnarray*}
Since $P_{\bf p}(0 \leftrightharpoons 0)=1$, 
\begin{eqnarray*}
P_{\bf p}(0 \leftrightharpoons 1) 
&=& \left( \frac{p_1}{2}+p_2 \right) + \left(p_0+ \frac{p_1}{2}\right) \left( \frac{p_1}{2}+p_2 \right), \\ 
P_{\bf p}(0 \leftrightharpoons 2) 
&=& \left( \frac{p_1}{2}+p_2 \right) P_{\bf p}(1 \leftrightharpoons 2)  
 + \left(p_0+ \frac{p_1}{2}\right) \left( \frac{p_1}{2}\left( \frac{p_1}{2}+p_2 \right) +p_2 \right) \\ 
&=& \left( p_1 +p_2 \right) P_{\bf p}(0 \leftrightharpoons 1)  
 - \left( \left( \frac{p_1}{2}\right) ^2 -p_0p_2\right) P_{\bf p}(0 \leftrightharpoons 0),
\end{eqnarray*}
the above recurrence formula is also valid for $n=2$. 
Using the characteristic roots  
\begin{eqnarray*}
\alpha = \frac{p_1+p_2}{2} + \sqrt{\left( p_0 + \frac{p_1}{2} + \frac{p_2}{4} \right) p_2}, \ 
\beta = \frac{p_1+p_2}{2} - \sqrt{\left( p_0 + \frac{p_1}{2} + \frac{p_2}{4} \right) p_2}
\end{eqnarray*}
of the characteristic equation of the recurrence relation, we have 
\begin{eqnarray*}
 P(0 \leftrightharpoons n) = \frac{\alpha^n(P_{\bf p}(0 \leftrightharpoons 1)-\beta) - \beta^n(P_{\bf p}(0 \leftrightharpoons 1)-\alpha) }{\alpha-\beta}\end{eqnarray*}
for $n\ge 2$ and this is also valid for $n=0,1$. 
Since the absolute values of $\alpha, \beta$ are less than $1$ if $p_2 \not= 1$, 
we have 
\begin{eqnarray*}
\chi(T_2,{\bf p}) 
&=& \mathbb{E}_{\bf p}(|C|) 
= \sum_{n=-\infty}^{\infty}P_{\bf p}(0 \leftrightharpoons n)
= 1+ 2\sum_{n=1}^{\infty}P_{\bf p}(0 \leftrightharpoons n) \\ 
&=&  1+ 2 \left( \frac{P_{\bf p}(0 \leftrightharpoons 1)-\beta}{\alpha-\beta}\sum_{n=1}^{\infty}\alpha ^n - 
\frac{(P_{\bf p}(0 \leftrightharpoons 1)-\alpha) }{\alpha-\beta} \sum_{n=1}^{\infty} \beta^n \right) \\ 
&=&  1+ 2 \left( \frac{P_{\bf p}(0 \leftrightharpoons 1)-\beta}{\alpha-\beta}\frac{\alpha}{1-\alpha} - 
\frac{(P_{\bf p}(0 \leftrightharpoons 1)-\alpha) }{\alpha-\beta} \frac{\beta}{1-\beta} \right) \\ 
&=&  1+ \frac{ 2\left( P_{\bf p}(0 \leftrightharpoons 1)-\alpha\beta \right) }{(1-\alpha)(1-\beta)}\\ 
&=&  1+ \frac{ 2\left( \left( \frac{p_1}{2}+p_2 \right) + \left(p_0+ \frac{p_1}{2}\right) \left( \frac{p_1}{2}+p_2 \right) 
- \left( \left( \frac{p_1}{2}\right) ^2 -p_0p_2\right) \right) }{\left( 1-(p_1+p_2) + \left( \left( \frac{p_1}{2}\right) ^2 -p_0p_2\right)\right)}\\ 
&=&  \frac{ 8(1+ p_0)}{\left( 2p_0 + p_1\right)^2} - 3. 
\end{eqnarray*}

On the other hand, since $P_{\bf p}(0 \leftrightarrow n)=\left( p_2+\frac {p_1}{2} \right)^2p_2^{n-1}$ for $n\ge 1$, we have 
\begin{eqnarray*}
\tilde{\chi}(T_2,{\bf p}) 
&=& \mathbb{E}_{\bf p}(|\tilde{C}|) 
= \sum_{n=-\infty}^{\infty}P_{\bf p}(0 \leftrightarrow n)
= 1+ 2\sum_{n=1}^{\infty}P_{\bf p}(0 \leftrightarrow n) \\ 
&=&  1+ 2\left( p_2+\frac {p_1}{2} \right)^2\frac{1}{p_2}\sum_{n=1}^{\infty} p_2^{n} \\ 
&=&  1+ 2\left( p_2+\frac {p_1}{2} \right)^2\frac{1}{p_2}\frac{p_2}{1-p_2} \\ 
&=& \frac{(2p_2+p_1)^2 }{2(1-p_2)} + 1. 
\end{eqnarray*}
\end{proof}

\section{$d$-Regular Tree ($d\ge 3$)} \label{sec:tree}

Let $T_d$ be the $d$-regular directed tree with the origin $o$, that is, a connected $d$-regular directed graph without cycles.
We assume $d\ge 3$. 
The following proof relies on the comments by Demeter Kiss. 

\begin{theorem}\label{thm:weaktree}
If $p_1=1-p$ and $p_2= p$, then 
\begin{eqnarray*}
\theta(T_d,{\bf p})>0 \ \ \ \ \text{ if and only if } \ \ \ \ p> p_{c,2}(T_d) 
\end{eqnarray*}
where 
\begin{eqnarray*}
p_{c,2}(T_d) :=  \frac{1}{(d^2-d-1)+\sqrt{(d^2-d-1)^2 - (d-1)^2}} \approx \frac{1}{2d^2} 
\end{eqnarray*}
is a decreasing function in $d$, satisfies $0<p_{c,2}(T_d)<1$. 
 
In particular, if $p_0+p_1 = 1$ then $\theta(T_d,{\bf p})=0$, and if $p_2+\dots +p_d = 1$ then $\theta(T_d,{\bf p})>0$.  
\end{theorem}

\begin{proof} 
First we assume $0<p<1$. 
Let $\ell(x)$ be the distance of $x$ from $o$. 
We say $x \in V$ is in the $k$-th level if $k=\ell(x)$. 
The vertices $x$ in $T_d$ without the origin have four types: we say 
\begin{enumerate} 
\item $x$ is of type 1 if there is $(x,y) \in E_\omega$ such that $\ell(y) = \ell(x) - 1$, and there is no $(x,z) \in E_\omega$ such that $\ell(z) = \ell(x)+1$; 
\item $x$ is of type 2 if there is $(x,y) \in E_\omega$ such that $\ell(y) = \ell(x) - 1$, and there is $(x,z) \in E_\omega$ such that $\ell(z) = \ell(x)+1$; 
\item $x$ is of type 3 if there is no $(x,y) \in E_\omega$ such that $\ell(y) = \ell(x)-1$, and there is only one $(x,z) \in E_\omega$ such that $\ell(z)=\ell(x)+1$; 
\item $x$ is of type 4 if there is no $(x,y) \in E_\omega$ such that $\ell(y)=\ell(x)-1$, and there is two $(x,z) \in E_\omega$ such that $\ell(z)=\ell(x)+1$.  
\end{enumerate}
Then for $x \in V\backslash \{o\}$, the probability that $v$ is of type 1 is $(1-p)\frac{1}{d}$, of type 2 is $p\frac{2}{d}$, 
of type 3 is $(1-p)\frac{d-1}{d}$, of type 4 is $=p\frac{d-2}{d}$.  

Let $m_{j i}$ be the expected number of type $j$ vertices weakly connected to a fixed type $i$ vertex in one low level. 
Then we can calculate $m_{j i}$ as:  
\begin{eqnarray*}
m_{1i}  
= \sum_{i=1}^{d-1} i \binom{d-1}{i}\left( (1-p)\frac{1}{d} \right)^i \left( 1- (1-p)\frac{1}{d}\right)^{(d-1)-i}  
= (1-p)\frac{d-1}{d}, 
\end{eqnarray*}
\begin{eqnarray*}
m_{2i}  
= \sum_{i=1}^{d-1} i \binom{d-1}{i}\left( p\frac{2}{d} \right)^i \left( 1- p\frac{2}{d}\right)^{(d-1)-i}  
= 2p\frac{d-1}{d}, 
\end{eqnarray*}
for $i=1,2,3,4$, 
$m_{31} = m_{41} = 0$, $m_{32} = m_{33}  = (1-p)\frac{d-1}{d}$, $m_{34} = 2(1-p)\frac{d-1}{d}$, $m_{42} = m_{43}  = p\frac{d-2}{d}$, $m_{44} = 2p\frac{d-2}{d}$. 
Let 
\begin{eqnarray*}
M 
= (m_{ji}) 
= 
\left(
    \begin{array}{c c c c}
      (1-p)\frac{d-1}{d} & (1-p)\frac{d-1}{d}   & (1-p)\frac{d-1}{d}   & (1-p)\frac{d-1}{d}    \\
      2p\frac{d-1}{d}    & 2p\frac{d-1}{d}      & 2p\frac{d-1}{d}      & 2p\frac{d-1}{d}       \\
      0                  & (1-p)\frac{d-1}{d} & (1-p)\frac{d-1}{d} & 2(1-p)\frac{d-1}{d} \\
      0                  & p\frac{d-2}{d}     & p\frac{d-2}{d}     & 2p\frac{d-2}{d}     \\
   \end{array}
 \right). 
\end{eqnarray*}
Since every elements in $M^2$ are positive for $0<p<1$, $M$ is strictly positive. 
Because some vertex weakly connected to more than two vertices with other types in one high level, $M$ is non-singular. 
Let $q^{(i)} $ be the probability of eventual extinction of the process initiated with a single vertex of type $i$ ($i=1,2,3,4$).  
Then we can use the following 
\begin{theorem}[\cite{MR0163361}]\label{thm:Harris}
Assume $M$ is strictly positive and non-singular. 
Let $\rho $ be the maximum eigenvalue of $M$. 
\begin{enumerate}
\item If $\rho \le 1$, $q^{(i)}=1$ for all $i=1,2,3,4$. 
\item If $\rho > 1$, $q^{(i)}<1$ for some $i=1,2,3,4$. 
\end{enumerate}
\end{theorem}
Since the origin is weakly connected to some vertex $x$ in 1-st level, and $P(v $ is type $i)>0$ for all $i=1,2,3,4$ if $0<p<1$, 
we see that $\rho > 1$ if and only if $\theta(T_d,{\bf p}) > 0$. 
The eigenvalues of $M$ are $0,0$ and 
\begin{eqnarray*}
 \frac{(d-2) p+(d-1) \pm \sqrt{(3-2d) p^2+2(d-1)^2 p}}{d}. 
\end{eqnarray*}
Since $p\le 1$ and $d\ge 3$, the maximum eigenvalue $\rho $ of $M$ is larger than $1$ if and only if 
\begin{eqnarray*} 
p 
>  \frac{1}{(d^2-d-1) + \sqrt{(d^2-d-1)^2 - (d-1)^2}}. 
\end{eqnarray*}
Since $\{|C|=\infty\}$ is an increasing event, we get the conclusion. 
\end{proof}

\begin{theorem}\label{thm:strongtree}
Let $2 \le k \le d-1$. 
If $p_k=1-p$, $p_{k+1}=p$, then 
\begin{eqnarray*}
\tilde{\theta}(T_d,{\bf p}) > 0  \ \ \ \ \text{ if and only if } \ \ \ \  p > \frac{d-k(k-1)}{2k}. 
\end{eqnarray*}
In particular, 
if $p_0+\dots +p_k=1$ then $\tilde{\theta}(T_d,{\bf p}) = 0$, and if $p_{k+1}+\dots +p_d=1$ then $\tilde{\theta}(T_d,{\bf p}) > 0$ for $k(k-1) \le d <  (k+1)k$. 
\end{theorem}

\begin{proof} 
First we assume $0<p<1$. 
The vertices $x$ strongly connected to the origin in $T_d$ without the origin have two types: we say 
\begin{enumerate} 
\item $x$ is of type 1 if there is $(x,y) \in E_\omega$ such that $\ell(y)=\ell(x)-1$ and there are $k-1$ edges $(x,z) \in E_\omega$ such that $\ell(z)=\ell(x)+1$; 
\item $x$ is of type 2 if there is $(x,y) \in E_\omega$ such that $\ell(y)=\ell(x)-1$ and there are $k$ edges $(x,z) \in E_\omega$ such that $\ell(z)=\ell(x)+1$. 
\end{enumerate}

Let $m_{j i}$ be the expected number of type $j$ vertices strongly connected to a fixed type $i$ vertex in one low level. 
Then we can calculate $m_{j i}$ as:  
\begin{eqnarray*}
m_{11} 
&=& \sum_{i=1}^{k-1} i \binom{k-1}{i} \left( p_k \frac{k}{d} \right)^i \left( 1- p_{k}\frac{k}{d}\right)^{(k-1)-i}  
= p_k \frac{k(k-1)}{d}, \\ 
m_{12} 
&=& p_k \frac{k^2}{d}, \ \ \ \ \ 
m_{21} 
= p_{k+1} \frac{(k+1)(k-1)}{d}, \ \ \ \ \ 
m_{22} 
= p_{k+1} \frac{(k+1)k}{d}, \ \ \ \ \ 
\end{eqnarray*}
Let 
\begin{eqnarray*}
M 
= 
\left(
    \begin{array}{c c}
      m_{11} & m_{12} \\
      m_{21} & m_{22} 
    \end{array}
\right)
= 
\left(
    \begin{array}{c c c c}
      p_k \frac{k(k-1)}{d} & p_k \frac{k^2}{d}   \\
      p_{k+1} \frac{(k+1)(k-1)}{d}            & p_{k+1} \frac{(k+1)k}{d}      \\
   \end{array}
 \right). 
\end{eqnarray*}
Since every elements in $M$ are positive for $0<p<1$, $M$ is strictly positive. 
Because some vertex strongly connected to more than two vertices with other types in one high level, then $M$ is non-singular. 
Since the probability that the origin is strongly connected some vertex $x$ in 1-st level is positive, and $P_{\bf p}(x $ is type i $)>0$ for each $i=1,2$, 
by Theorem \ref{thm:Harris}, we see that the maximum eigenvalue $\rho $ of $M$ is larger than $1$ if and only if $\tilde{\theta}(T_d,{\bf p}) > 0$. 
The maximum eigenvalue $\rho $ of $M$ is  
\begin{eqnarray*}
p_k \frac{k(k-1)}{d} + p_{k+1} \frac{(k+1)k}{d} 
=  \frac{k( 2p +k-1) }{d}.
\end{eqnarray*}
Hence $\rho > 1 $ if and only if 
\begin{eqnarray*}
p > \frac{d-k(k-1)}{2k} . 
\end{eqnarray*}
If $d < (k+1)k$, then $ \frac{d-k(k-1)}{2k} <1$. 
Since $\{|\tilde{C}|=\infty \}$ is an increasing event, if $p_{k+1}=1$ then $\tilde{\theta}(T_d,{\bf p}) > 0$. 
Similarly, if $d > k(k-1)$, then $0 < \frac{d-k(k-1)}{2k}$. 
Hence if $p_k=1$, then $\tilde{\theta}(T_d,{\bf p}) = 0$.

If $p_k=1$ with $d=k(k-1)$, then the vertices $x \in V\backslash \{o\}$ strongly connected to the origin in $T_d$ have the property that there is $(x,y) \in E_\omega$ such that $\ell(y)=\ell(x)-1$ and there are $k-1$ edges  $(x,z) \in E_\omega$ such that $\ell(z)=\ell(x)+1$. 
Hence the the expected number of vertices strongly connected to a vertex in one low level is $1$. 
By theory of branching process, we conclude $\tilde{\theta}(T_d,{\bf p}) = 0$. 
\end{proof}

\section{Sufficient conditions of $\theta(G,{\bf p})=0$, $\tilde{\theta}(G,{\bf p})=0$} \label{sec:FiniteCluster}

We saw, for the regular tree $T_d$ with $d\ge 3$, if $p_0+p_1 = 1$ then $\theta(T_d,{\bf p})=0$, and if $p_k=1$ for $k(k-1) \le d$ then $\tilde{\theta}(T_d,{\bf p}) = 0$. 
We prove these are valid for every $G$ with $d\ge 3$ except $d = k(k-1) $. 

\begin{theorem}\label{thm:near0}
Assume $d\ge 3$. 
\begin{enumerate}
\item \label{item:1} If $p_0+p_1$ is sufficiently near to $1$, then $ \theta(G,{\bf p}) =0$.
\item \label{item:2}  
If $k(k-1)<d$ and $p_0+p_1+\cdots+p_k$ is sufficiently near to $1$, then $ \tilde{\theta}(G,{\bf p}) =0$.
\end{enumerate}
\end{theorem}

\begin{proof}
\noindent
(\ref{item:1})
For $x\in V$ and $y, z \in L_x$ with $y \not= z$, 
let 
\begin{eqnarray*}
p_0'(d) 
&:=& P_{\bf p}((x,y) \not \in E_\omega \text{ and } (x,z) \not \in E_\omega ) \\ 
&=& p_0+\frac{d-2}{d}p_1+\sum_{i=2}^{d-2} p_i\frac{(d-i)(d-i-1)}{d(d-1)} \\
\frac{p_1'(d)}{2}
&:=& P_{\bf p}((x,y) \in E_\omega \text{ and } (x,z) \not \in E_\omega ) 
= \frac{1}{d}p_1+\sum_{i=2}^{d-1} p_i\frac{i(d-i)}{d(d-1)} \\ 
p_2'(d) 
&:=&  P_{\bf p}((x,y) \in E_\omega \text{ and } (x,z) \in E_\omega ) 
= \sum_{i=2}^d p_i\frac{i(i-1)}{d(d-1)} \\ 
q(d) 
& :=&  P_{\bf p}((x,y) \in E_\omega \text{ or } (y,x) \in E_\omega ) 
>0. 
\end{eqnarray*}
Then as the proof of Theorem \ref{thm:2-regular}, for a self-avoiding path $(x_0,x_1, \dots ,x_n)$ we have 
\begin{eqnarray*}
P_{\bf p} ( \cap_{i=1}^{n}\{ (x_{i-1},x_i)\in E_\omega \text{ or } (x_i,x_{i-1}) \in E_\omega \} ) 
= \frac{\alpha^n(q(d)-\beta) - \beta^n(q(d)-\alpha) }{\alpha-\beta}, 
\end{eqnarray*}
where 
\begin{eqnarray*}
\alpha &=& \frac{p_1'(d)+p_2'(d)}{2} + \sqrt{\left(p_0'(d) + \frac{p_1'(d)}{2}+\frac{p_2'(d)}{4} \right)p_2'(d)}, \\ 
\beta &=& \frac{p_1'(d)+p_2'(d)}{2} - \sqrt{\left(p_0'(d) + \frac{p_1'(d)}{2}+\frac{p_2'(d)}{4} \right)p_2'(d)}. 
\end{eqnarray*}
Let $\Gamma (n)$ be the set of self-avoiding paths of $G$ with the length $n$ and beginning at the origin, and  
$\Gamma' (n)$ the set of self-avoiding paths of $G$ with the length $n$ and bargaining at the origin such that each edge is weakly connected, 
$N(n)$ the number of the elements in $\Gamma' (n)$. 
Since for $n$ there is $\epsilon = \epsilon(n) >0$ such that $\sigma_G (n) \le (\lambda (G) +\epsilon(n))^n  $ and $\lim_{n \to \infty} \epsilon (n)=0$, we have 
\begin{eqnarray*}
\theta(G,{\bf p}) 
&\le & P_{\bf p} (N(n)\ge 1)  
\le \mathbb{E}_{\bf p} (N(n)) 
= \sum_{\gamma \in \Gamma(n)} P_{\bf p} (\gamma \in \Gamma' (n)) \\
&\le & \sigma_G (n) \frac{\alpha^n(q(d)-\beta) - \beta^n(q(d)-\alpha) }{\alpha-\beta} \\ 
&\le &  ((\lambda (G) + \epsilon(n))\alpha)^n \frac{|q(d)-\beta| }{\alpha-\beta} 
+ ((\lambda (G) +\epsilon(n))|\beta|)^n \frac{|q(d)-\alpha| }{\alpha-\beta}. 
\end{eqnarray*}
Since $|\beta|<\alpha $, if $(\lambda (G) + \epsilon(n))\alpha<1 $ for every sufficiently large $n$, then $ \theta(G,{\bf p}) =0$. 
Thus if $\lambda (G)\alpha <1 $, then $ \theta(G,{\bf p}) =0$. 
If $p_0+p_1$ is sufficiently near to $1$, then we obtain 
\begin{eqnarray*}
\alpha 
&=& \frac{p_1'(d)+p_2'(d)}{2} + \sqrt{\left(p_0'(d) + \frac{p_1'(d)}{2}+\frac{p_2'(d)}{4} \right)p_2'(d)} \\ 
&\le & \frac{1}{d}p_1 + \sum_{i=2}^d p_i  + \sqrt{\sum_{i=2}^d p_i} \\ 
&\le & \frac{1}{d}p_1 + (1-(p_0+p_1))  + \sqrt{1-(p_0+p_1)} \\ 
&<& \frac{1}{d-1} \le \frac{1}{\lambda(G)}.  
\end{eqnarray*}
Hence if $p_0+p_1$ is sufficiently near to $1$, then $ \theta(G,{\bf p}) =0$.

\noindent
(\ref{item:2})
For each self-avoiding path $\gamma =(x_0, x_1, \dots ,x_n)$ with length $n$, we have 
\begin{eqnarray*}
&& P_{\bf p} ( (x_{i-1},x_i) \in E_\omega \text{ and } (x_i,x_{i-1}) \in E_\omega \text{ for } i=1,2,\dots ,n) \\
&=& \left(\sum_{i=1}^d p_i\frac{\binom{d-1}{i-1}}{\binom{d}{i}} \right) \left(\sum_{i=2}^d p_i\frac{\binom{d-2}{i-2}}{\binom{d}{i}} \right)^{n-1}  
\left(\sum_{i=1}^d p_i\frac{\binom{d-1}{i-1}}{\binom{d}{i}} \right)  \\ 
&=& p_2'(d)^{n-1} \left(\sum_{i=1}^d p_i\frac{i}{d} \right)^2.  
\end{eqnarray*}
Let $\tilde{\Gamma}' (n)$ be the set of self-avoiding paths of $G$ with the length $n$ and bargaining at the origin such that each edge is strongly connected, and $\tilde{N}(n)$ the number of the elements in $\tilde{\Gamma}' (n)$. 
Then we get 
\begin{eqnarray*}
\tilde{\theta}(G,{\bf p}) 
&\le & P_{\bf p} (\tilde{N}(n)\ge 1)  
\le \mathbb{E}_{\bf p} (\tilde{N}(n)) 
= \sum_{\gamma \in \Gamma(n)} P_{\bf p} (\gamma \in \tilde{\Gamma}' (n)) \\
&\le & \sigma_G (n) p_2'(d)^{n-1} \left(\sum_{i=1}^d p_i\frac{i}{d} \right)^2\\ 
&\le & \left( (\lambda (G) +\epsilon(n)) p_2'(d) \right)^{n-1} (\lambda (G) +\epsilon(n)) \left(\sum_{i=1}^d p_i\frac{i}{d} \right)^2. 
\end{eqnarray*}
Thus if $\lambda (G) p_2'(d) <1 $, then $ \tilde{\theta}(G,{\bf p}) =0$. 
If $k(k-1)<d$  and $p_0+p_1+\cdots+p_k$ is sufficiently near to $1$, then we obtain 
\begin{eqnarray*}
p_2'(d) 
&=& \frac{1}{d-1}\sum_{i=2}^k p_i\frac{i(i-1)}{d} + \sum_{i=k+1}^d p_i \frac{i(i-1)}{d(d-1)} \\ 
&\le & \frac{1}{d-1}\frac{k(k-1)}{d} + (1-(p_0+p_1+\cdots+p_k )) \\  
&<& \frac{1}{d-1} \le \frac{1}{\lambda(G)}. 
\end{eqnarray*}
Hence if $p_0+p_1+\cdots+p_k$ is sufficiently near to $1$, then $ \theta(G,{\bf p}) =0$.
\end{proof}

\section{Sufficient conditions of $\theta(G,{\bf p})>0$, $\tilde{\theta}(G,{\bf p})>0$}\label{sec:InfiniteCluster}

A {\it directed planar graph} is a pair $G=(V,E)$ such that a countable set $V \subset \mathbb{R}^2$ and 
\begin{eqnarray*} 
E \subset \{ e : [0,1] \to \mathbb{R}^2 \mid \text{ continuous, }e(0),e(1)\in V, e(0) \not= e(1) \}
\end{eqnarray*}
satisfying $e=e'$ in $E$ if $e(0)=e'(0) $, $e(1)=e'(1)$, and for any $e\in E$ there is $e^{-1}\in E$ such that $e^{-1}(t)=e(1-t)$ for all $t\in [0,1]$, and 
$e((0,1)) \cap e'((0,1)) =\emptyset $ for all $e,e'\in E$ with $e' \not= e$ and $e' \not= e^{-1}$.
We can regard a directed planar graph as a directed graph by considering $e(0)$ and $e(1)$ to be the ordered pair of vertices. 
We assume directed planar graphs are $d$-regular ($d\ge 3$) and vertex transitive in the sense of a directed graph. 
Sometime we denote the image $e([0,1])$ by $e$. 

Let $\mathcal{M}$ be the set of all connected components in $\mathbb{R}^2 \backslash \cup_{e\in E} e$. 
For each $M\in \mathcal{M}$, we choose a point $x^*=x^*(M) \in M$, and set $V^*:= \{x^*(M) \mid M \in \mathcal{M}\}$. 
For each $e \in E$, there are connected components $M_e(0), M_e(1) \in \mathcal{M}$ such that $e= \overline{M_e(0)}\cap \overline{M_e(1)} $, where $\overline{M_e(i)} $ is the closure of $M_e(i)$, and the direction of $e$ is corresponding to the direction of the counter-clockwise rotation of the boundary of $M_e(0)$ and the direction of the clockwise rotation of the boundary of $M_e(1)$. 
For each $e \in E$, we give a continuous map $e^*:[0,1] \to \overline{M_e(0)} \cup \overline{M_e(1)}$ such that 
$e^*(0)=x^*(M_e(0))$, $e^*(1)=x^*(M_e(1))$, and 
$(e^{-1})^* = (e^*)^{-1}$, 
$e^*((0,1))\cap (e')^*((0,1)) =\emptyset $ for all $e^*,(e')^* $ with $e' \not= e$ and $e' \not= e^{-1}$.
Then the directed planar graph $G^* =(V^*,E^*)$ is called the {\it dual graph} of $G$.  
The map $m:E\to E^*$ defined by $m(e)=e^*$ is bijective.  

A {\it self-avoiding polygon} ({\it SAP}) with length $n$ in $G$ is a path $(x_0, x_1, \dots ,x_{n-1})$ such that $x_{n-1} \sim x_0$ and $x_i \sim x_{i+1}$ for all $i=0,1,\dots ,n -2$, and $x_i \not= x_j$ for all $i \not= j$. 
For $k \ge 1$, let $N(k) := \min\{n \mid $ there is a self-avoiding polygon in $G^*$ with the length $n$ such that $B(k) $ in $G$ is enclosed$ \}$. 
The proof of the theorem below is based on the arguments in \cite{MR1707339} for the proof of the critical probability is less than 1. 

\begin{proposition}\label{prop:b(p)}
Assume $N(k) \to \infty$ as $k \to \infty$. 
Let 
\begin{eqnarray*}
b({\bf p}) 
:= \max_{i=1,2,\dots,d-1} \left\{ \left( \sum_{j=0} ^{d-i} p_j\prod_{k=0}^{i-1} \frac{(d-k)-j}{d-k} \right)^{\frac{1}{i}}\right\}.  
\end{eqnarray*}  
\begin{enumerate}
\item\label{item:b2} If $\lambda (G^*) b({\bf p})^2<1$, then $\theta(G,{\bf p})>0$. 
\item\label{item:b1} If $\lambda (G^*) b({\bf p})<1$ and $p_0+p_1<1$, then $\tilde{\theta}(G,{\bf p})>0$. 
\end{enumerate}
\end{proposition}

\begin{proof} 
\noindent 
(\ref{item:b2}) 
If $p_0=1$, then we have $b({\bf p})=1$. 
Since $\lambda (G^*) \ge 1$, the assumption $\lambda (G^*) b({\bf p})^2<1$ gives $p_0<1$, that is, $p_1+\dots+p_d>0$. 
For a subgraph $G_\omega=(V,E_\omega)$ of $G$ we define the subgraph $G_\omega^*=(V^*, E^*_\omega)$ in $G^*$ by $ E^*_\omega=m(E_\omega) $. 

\begin{lemma}
For $n,k\in \mathbb{N}$ let $\gamma^* $ be a self-avoiding polygon in $G^*$ with length $n$ such that $B(k) $ in $G$ is enclosed. 
Then 
\begin{eqnarray*}
P_{\bf p}(\text{every edge on } \gamma^* \text{ are not in } E^*_\omega) \le b({\bf p})^{2n}.
\end{eqnarray*}
\end{lemma}	

\begin{proof}
Each edge $e^*$ on $\gamma^*$ has two neighbor vertices $ x,y \in V$ such that $e^* \in \{ m (x,y),m( y,x)\}$. 
Let $W$ be the set of vertices in $V$ around $\gamma^* $.  
The event that every edge on $\gamma^*$ are not in $E^*_\omega$ depends only on the state of vertices in $W$.
If $x \in W$ has exactly $i$ neighbor edge on $\gamma^* $ ($1\le i \le d-1$), then 
\begin{eqnarray*}
P_{\bf p}(m(x,y) \not \in \gamma^*  \text{ for any } y \in \omega(x)) 
= \sum_{j=0} ^{d-i} p_j\frac{\binom{d-i}{j}}{\binom{d}{j}} 
= \sum_{j=0} ^{d-i} p_j\prod_{k=0}^{i-1} \frac{(d-k)-j}{d-k}.
\end{eqnarray*}
Hence 
\begin{eqnarray*}
P_{\bf p}(\text{every edge on } \gamma^* \text{ are not in } E^*_\omega) 
&=& \prod_{x \in W} P_{\bf p}(m(x,y) \not \in \gamma^*  \text{ for every } y \in \omega(x)) \\ 
&\le & \max_{i=1,2,\dots,d-1} \left\{ \left( \sum_{j=0} ^{d-i} p_j\prod_{k=0}^{i-1} \frac{(d-k)-j}{d-k} \right)^{\frac{1}{i}}\right\}^{2n}. 
\end{eqnarray*}
\end{proof}

For $k \ge 1$, let $M_k(n)$ be the set of self-avoiding polygons in $G^*$ with length $n$ such that $B(k)$ is enclosed and its edges are not in $ E^*_\omega$. 
Let $F_k$ be the event that there exists a self-avoiding polygon in $G^*$ such that $B(k)$ is enclosed and its edges are not in $E^*_\omega$. 
Since ${\bf p}$ satisfies $\lambda (G^*) b({\bf p})^2<1$, there is $\epsilon >0$ such that $(\lambda (G^*) +\epsilon)b({\bf p})^2<1$.  
The number of self avoiding polygons in $G^*$ having length $n$ and beginning at the origin is not greater than $\sigma_{G^*}(n-1)$. 
Using the inequality $\sum_{k=s}^{\infty} k z^k=\frac{s -(s-1)z}{(1-z)^2}z^s \le \frac{s}{(1-z)^2}z^s$ for $0<z<1$ and $s \ge 1$, 
we have 
\begin{eqnarray*}
P_{\bf p}(F_k)
& = & P_{\bf p} \left( |M_k(n)| \ge 1 \text{ for some } n\right) \\ 
&\le & \sum_{n=N(k)}^\infty \sum_{\gamma^*:\text{SAP in }G^* } P_{\bf p}(\gamma^* \in M_k(n)) \\   
&\le & \sum_{n=N(k)}^\infty n \sigma_{G^*} (n-1) b({\bf p})^{2n} \\
&\le & \frac{1}{(\lambda (G^*) +\epsilon)} \sum_{n=N(k)}^\infty n \left((\lambda (G^*) +\epsilon)b({\bf p})^2\right)^n \\ 
&\le & \frac{N(k)\left((\lambda (G^*) +\epsilon)b({\bf p})^2\right)^{N(k)}}{(\lambda (G^*) +\epsilon)(1-(\lambda (G^*) +\epsilon)b({\bf p})^2)^2} \\ 
&<& \frac{1}{2} 
\end{eqnarray*}
for sufficiently large $k$. 
The event $F_{k+1}$ and hence $F_{k+1}^c := \Omega \backslash F_{k+1}$ are depend only on $V\backslash B(k)$. 
For each $\omega \in F_{k+1}^c $ there exists a vertex $x \in B(k+1)$ such that there exists a self-avoiding weakly connected infinite path on $V\backslash B(k)$ beginning at $x$.  
(By giving an order on $V$, we can determine the vertex $x$ unique.) 
There is a path $(z_0=o,z_1,\dots, z_k, z_{k+1}=x) $ from the origin $o$ to $x$ in $B(k) \cup \{x\}$ with the length $k+1$.  
The event that $u_{\ell+1} \in \omega(u_\ell) $ for $\ell=0,1,\dots ,k$ depends only on $B(k)$. 
The probability of this event is $\left( \sum_{i=1}^d\frac{i}{d} p_i \right)^{k+1}>0$. 
Hence
\begin{eqnarray*}
P_{\bf p}(|C|=\infty) 
&\ge& P_{\bf p}(F_{k+1}^c \cap \{ x \leftrightharpoons o\})  \\ 
&=& P_{\bf p}(F_{k+1}^c \cap \{ x \leftrightharpoons o\}|F_{k+1}^c)P_{\bf p}(F_{k+1}^c)  \\ 
&\ge & \left( \sum_{i=1}^d\frac{i}{d} p_i \right)^{k+1}(1- P_{\bf p} (F_{k+1})) \\
&>& 0. 
\end{eqnarray*}

\noindent 
(\ref{item:b1}) 
For a self-avoiding polygon $\gamma^* $ in $G^*$ with length $n$ such that $B(k) $ is enclosed, 
as the case of (\ref{item:b2}), 
observing the set of vertices around $\gamma^* $ such that enclosed by $\gamma^* $, we have 
\begin{eqnarray*}
P_{\bf p}(\text{every edge on } \gamma^* \text{ are not strongly connected} ) \le b({\bf p})^{n}.
\end{eqnarray*}
Let $\tilde{F}_k$ be the event that there exists a self-avoiding polygon in $G^*$ such that $B(k)$ is enclosed and its edges are not  strongly connected. 
Then we have $P_{\bf p}(\tilde{F}_k)< \frac{1}{2} $ for sufficiently large $k$ by the same proof of (\ref{item:b2}).  
The event $\tilde{F}_{k+1}$ and hence $\tilde{F}_{k+1}^c := \Omega \backslash \tilde{F}_{k+1}$ are depend only on $V\backslash B(k)$. 
For each $\omega \in \tilde{F}_{k+1}^c $ there exists a vertex $\tilde{x} \in B(k+1)$ such that there exists a self-avoiding strongly connected infinite path $(\tilde{x}, \tilde{y},\dots )$ beginning at $\tilde{x}$ on $V\backslash B(k)$. 
(By giving an order on $V$, we can determine the vertex $\tilde{x}$ unique.) 
There is a path $\tilde{\gamma} = (\tilde{z}_0 = o,\tilde{z}_1,\dots,\tilde{z}_k,\tilde{z}_{k+1}=\tilde{x}) $ from the origin $o$ to $\tilde{x}$ in $B(k) \cup \{\tilde{x}\}$ with the length $k+1$.  
We fix $\ell \ge 2$ with $p_\ell >0$. 
For $\omega \in \tilde{F}_{k+1}^c$ we define $\omega' $ as 
\begin{eqnarray*}
\omega'(v) = 
\begin{cases}
\omega(v) \ \ \text{ if } \ v \not\in \tilde{\gamma},  \\
\{ \tilde{z}_k , \tilde{y}, \tilde{x}(j^x_1), \dots , \tilde{x}(j^x_{\ell-2}) \} \ \ \text{ if } \ v = \tilde{x}, \\ 
\{ \tilde{z}_{i-1} ,  \tilde{z}_{i+1}, \tilde{z_i}(j^i_1), \dots , \tilde{z_i}(j^i_{\ell-2}) \} \ \ \text{ if } \ v = \tilde{z_i} \text{ for } i=1,2,\dots ,k, \\ 
\{ \tilde{z}_1 ,  \tilde{z_i}(j^1_1), \dots , \tilde{z_i}(j^1_{\ell-1}) \} \ \ \text{ if } \ v = o \\ 
\end{cases}
\end{eqnarray*}
where $ j^i_h$ can be determined unique by using the order on $L_{\tilde{z}_i}$.  
We set $\tilde{G}_{k+1}:=\{\omega' \mid \omega \in \tilde{F}_{k+1}^c \}$. 
Then 
\begin{eqnarray*}
P_{\bf p}(|\tilde{C}|=\infty) 
&\ge& P_{\bf p}(\tilde{F}_{k+1}^c  \cap \{ \tilde{x} \leftrightarrow o\}) \\  
&\ge & P_{\bf p}(\tilde{G}_{k+1})   \\ 
&\ge & \left( \frac{\ell ! (d-\ell)!}{d!}p_\ell \right)^{k+2} P_{\bf p}(\tilde{F}_{k+1}^c)  \\ 
&\ge & \left( \frac{\ell ! (d-\ell)!}{d!}p_\ell \right)^{k+2} (1- P_{\bf p} (\tilde{F}_{k+1}))  \\ 
&>& 0. 
\end{eqnarray*}

\end{proof}

\begin{theorem}
Assume $N(l) \to \infty$ as $l \to \infty$. 
\begin{enumerate}
\item If $d =k $ or $\lambda(G^*)  < \left(\frac{d}{d-k}\right)^2$, 
and $p_k+\cdots + p_d$ is sufficiently near to $1$, then $\theta(G,{\bf p})>0$. 
\item If $d =k $ or 
$\lambda(G^*)  < \frac{d}{d-k} $, 
and $p_k+\cdots + p_d$ is sufficiently near to $1$, then $\tilde{\theta}(G,{\bf p})>0$. 
\end{enumerate}  
\end{theorem}
 
\begin{proof}
If $p:=p_k+\cdots + p_d$ is sufficiently near to $1$, then we have 
\begin{eqnarray*}
b({\bf p}) 
&=& \max_{i=1,2,\dots,d-k} \left\{ \left(\sum_{j=0} ^{k} p_j \prod_{l=0}^{i-1} \frac{(d-l)-j}{d-l}  + \sum_{j=k} ^{d-i} p_j \prod_{l=0}^{i-1} \frac{(d-l)-j}{d-l} \right)^{\frac{1}{i}}\right\} \\ 
&\le & \max_{i=1,2,\dots,d-k} \left\{ \left( (1-p) + \left( \sum_{j=k} ^{d-i} p_j \right) \max_{j=k,\dots ,d-i} \prod_{l=0}^{i-1} \frac{(d-l)-j}{d-l} \right)^{\frac{1}{i}}\right\} \\ 
&\le & \max_{i=1,2,\dots,d-k} \left\{ \left( (1-p) + p\prod_{l=0}^{i-1} \frac{(d-l)-k}{d-l} \right)^{\frac{1}{i}}\right\} \\ 
&=& \max_{i=1,2,\dots,d-k} \left\{ \left( (1-p) + p\prod_{l=0}^{k-1} \frac{(d-l)-i}{d-l} \right)^{\frac{1}{i}}\right\} . 
\end{eqnarray*}
We get 
\begin{eqnarray*}
\max_{i=1,2,\dots,d-k} \left\{ \left( \prod_{l=0}^{k-1} \frac{(d-l)-i}{d-l} \right)^{\frac{1}{i}}\right\}
&\le& \prod_{l=0}^{k-1} \max_{i=1,2,\dots,d-k} \left\{ \left( 1-\frac{i}{d-l} \right)^{\frac{1}{i}}\right\} \\ 
&=& \prod_{l=0}^{k-1} \left( 1-\frac{1}{d-l} \right) \\ 
&=& \frac{d-k}{d} . 
\end{eqnarray*}
Hence if $d =k $ or $\lambda(G^*)  < \left(\frac{d}{d-k}\right)^2$, then 
$\lambda (G^*) b({\bf p})^2 <1$ when $p$ is sufficiently near to $1$.
Similarly, if $d =k $ or $\lambda(G^*)  < \frac{d}{d-k}$, then $\lambda (G^*) b({\bf p}) <1$ when $p$ is sufficiently near to $1$.
\end{proof}

Using $\lambda(G^*)\le d^*-1$ for the degree $d^*$ of $G^*$, which is the maximum degree among vertices in $G^*$, we obtain 
\begin{corollary}\label{cor:dualedge} \ 
\begin{enumerate}
\item Assume $G$ and $G^*$ satisfies $d=3$ and $d^*\le 9$, or $d=4$ and $d^*\le 4$, or $d=5,6$ and $d^*\le 3$. \\ 
If $p_2 + \cdots + p_d$ is sufficiently near to $1$, then  $\theta(G,{\bf p})>0$. 
\item Assume $G$ and $G^*$ satisfies $d=3$, or $d=4$ and $d^*\le 4$, or $d=5$ and $d^*\le 3$.  \\ 
If $p_3 +\cdots + p_d$ is sufficiently near to $1$, then  $\tilde{\theta}(G,{\bf p})>0$. 
\end{enumerate}
\end{corollary}

The following directed planar graphs satisfy the assumption in Corollary \ref{cor:dualedge}: 
\begin{itemize}
\item The square lattice $\mathbb{L}^2$ is the directed graph $(V,E)$ with $V=\mathbb{Z}^2$ and $(x,y)\in E $ if and only if $\|x-y\|=1$, where $\|\cdot \|$ is the Euclidean norm. Note that $(\mathbb{L}^2)^*=\mathbb{L}^2$. 
\item The hexagonal lattice is the directed graph with the vertex set $\{(1,0)n + (\frac{1}{2},\frac{\sqrt{3}}{2})m \mid (n,m) \in  \mathbb{Z}^2\} \backslash \{(2,0)n + (1,\sqrt{3})m \mid (n,m) \in  \mathbb{Z}^2\}$ and the edge set $E$ satisfying $(x,y)\in E $ if and only if $\|x-y\|=1$. 
Its dual graph is isomorphic to the triangular lattice below. 
\item The triangular lattice is the directed graph with the vertex set $\{(1,0)n + (\frac{1}{2},\frac{\sqrt{3}}{2})m \mid (n,m) \in  \mathbb{Z}^2\}$ and the edge set $E$ satisfying $(x,y)\in E $ if and only if $\|x-y\|=1$. 
Its dual graph is isomorphic to the hexagonal lattice. 
This satisfies the assumption in (i), but does not satisfy  the assumption in (ii).  
\item Semi-regular tessellation $4.8.8$, the vertex set is the vertices of tiles, and the edge set is the set of edges of tiles with the directions.  
\item Hyperbolic tilings by $k$-regular polygons with $k \le 9$ of the hyperbolic plane such that the number of tiles at a vertex is three, the vertex set is the vertices of tiles, and the edge set is the set of edges of tiles with the directions.  
\end{itemize} 

Using $\lambda(G^*)=\sqrt{2+\sqrt{2}}$ proved in \cite{MR2912714} for the hexagonal lattice $G^*$, we have 
\begin{corollary}
For the triangular lattice $G$, if $p_3 +\cdots + p_d$ is sufficiently near to $1$, then  $\tilde{\theta}(G,{\bf p})>0$. 
\end{corollary}

\section{Inequalities and Russo's formula} \label{sec:Ineq}

\subsection{Inequalities}

For our model, there are increasing events which the FKG inequality $P(A \cap B) \ge P(A)P(B) $ is not valid. 
For example, on the infinite directed $2$-regular graph $T_2$, 
we consider the increasing events $A := \{ -1\leftrightharpoons 0\} $ and  $B:= \{ 0 \leftrightharpoons 1 \} $.  
Let ${\bf p} =(0,1-p,p)$. 
Then $P(A)=P(B)=1-\left(\frac{1-p}{2}\right)^2$ and 
\begin{eqnarray*}
P(A \cap B) 
= P(-1 \leftrightharpoons 1) 
= 1- 2\left(\frac{1-p}{2}\right)^2
= 1- \frac{(1-p)^2}{2}. 
\end{eqnarray*}
Hence 
\begin{eqnarray*}
P(A)P(B) -P(A \cap B) 
&=& \left(1-\left(\frac{1-p}{2}\right)^2\right)^2 - \left( 1-\frac{(1-p)^2}{2}\right) \\ 
&=& \left(\frac{1-p}{2}\right)^4
> 0 
\end{eqnarray*}
for $0\le p <1$. 
Also for the increasing events $\tilde{A} := \{ -1\leftrightarrow 0\} $, $\tilde{B}:= \{ 0 \leftrightarrow 1 \} $ and ${\bf p} =(0,1-p,p)$,  
we have 
$P(\tilde{A})=P(\tilde{B})=\left(\frac{1-p}{2}+p\right)^2=\left(\frac{1+p}{2}\right)^2$ and 
\begin{eqnarray*}
P(\tilde{A} \cap \tilde{B}) 
= P(-1 \leftrightarrow 1) 
= p\left(\frac{1-p}{2}+p\right)^2 
=p\left(\frac{1+p}{2}\right)^2. 
\end{eqnarray*}
Hence 
\begin{eqnarray*}
P(\tilde{A})P(\tilde{B}) -P(\tilde{A} \cap \tilde{B}) 
&=& \left(\frac{1+p}{2}\right)^4-p\left(\frac{1+p}{2}\right)^2 \\
&=& \left(\frac{1+p}{2}\right)^2 \left( \frac{1+p^2}{4} \right)
> 0
\end{eqnarray*}
for all $0\le p \le1$. 
That is the FKG inequality is not valid for these increasing events in our model. 

For $K \subset V$ and $\omega \in \Omega$, we denote $$C(K,\omega):= \{\omega' \in \Omega \mid \omega'(x) = \omega(x) \text{ for all } x \in K\}.$$ 
For events $A,B \in \mathcal{F}$ which depends only on the states of the finite vertices $W \subset V$, 
we define $$A \Box B:=\{ \omega \in \Omega \mid \text{ there is }K \subset W \text{ such that } C(K,\omega ) \subset A, C(W \backslash K,\omega) \subset B\}.$$

\begin{theorem}[Reimer's ineqality \cite{MR1751301}]
For any events $A,B$ which depends only on the states of the finite vertices, we have 
\begin{eqnarray*}
P_{\bf p}(A \Box B) \le P_{\bf p}(A) P_{\bf p}(B). 
\end{eqnarray*}
\end{theorem}

Note that since $A\cap (\Omega \backslash B) \not \subset A\Box (\Omega \backslash B) $ for increasing events $A,B$ in our model, we can't deduce the FKG inequality from the Reimer's inequality. 

Let $Q:=\{{\bf p}=(p_0,p_1,\dots ,p_d) \mid p_0+p_1+\dots +p_d=1\} \subset \mathbb{R}^{d+1}$. 
For $\omega \in  \Omega $, $x\in V$ and $S \subset L_x$, we define 
\begin{eqnarray*}
\omega^x_S(y):=
\begin{cases}
S & \text{ if } y = x \\
\omega(y) & \text{ otherwise,  }
\end{cases}
\end{eqnarray*}
and $A^x_S:=\{ \omega \in \Omega \mid \omega^x_S \in A\} $ for an event $A$. 
Then the event $A^x_S$ is independent of the state at $x$.  

\begin{theorem}[Russo's formula]
Let $A $ be an event which depends only on the states of the finite vertices $W \subset V$. 
For ${\bf p}\in Q$ and a $C^1$ curve $\gamma = (\gamma_0,\gamma_1, \dots, \gamma_d):(-\epsilon,\epsilon) \to Q$ with parameter $t$ such that $\gamma (0)={\bf p}$, 
we have 
\begin{eqnarray*}
\frac{d}{d t} P_{\gamma(t)}(A) |_{t=0}
= \sum_{x \in W} \sum_{S \subset L_x}  \frac{d\gamma_{|S|}}{d t}(0) \frac{P_{\bf p}(A^x_S)}{\binom{d}{|S|}}  . 
\end{eqnarray*}
\end{theorem}

\begin{proof}
For $x \in W$ and $t\in (-\epsilon,\epsilon)$, let 
\begin{eqnarray*}
\overrightarrow{\gamma}_x(t):=(\gamma_y(t) =(\gamma_{y,0}(t),\gamma_{y,1}(t),\dots,\gamma_{y,d}(t)) \mid y \in V)
\end{eqnarray*}
such that  $\gamma_x(t) = \gamma(t) $ and $\gamma_y(t) \equiv {\bf p}$ if $y\not=x$. 
We set $$ P_{\overrightarrow{\gamma}_x(t)} = \prod_{y\in V} \mu_y(t) $$
where $\mu_y(t)$ is given by 
\begin{eqnarray*} 
\mu_y(t) (\omega(y) = \emptyset) = \gamma_{y,0}(t), \ \  
\mu_y(t) (\omega(y) = \{y(i_1), \dots, y(i_k) \}) = \frac{\gamma_{y,k}(t)}{\binom{d}{k}}
\end{eqnarray*}
for each $\{y(i_1), \dots, y(i_k) \} \subset L_y$.  
Since $P_{\overrightarrow{\gamma}_x (t)}(A^x_S)$ is independent of the state at $x$, we have 
\begin{eqnarray*}
P_{\overrightarrow{\gamma}_x(t)}(A) 
&=& P_{\overrightarrow{\gamma}_x(t)}\left(\sqcup_{S \subset L_x} (\{ \omega \mid \omega(x) = S\} \cap A^x_S)\right) \\ 
&=& \sum_{S \subset L_x}P_{\overrightarrow{\gamma}_x(t)}(\{ \omega \mid \omega(x) = S\})P_{\overrightarrow{\gamma}_x(t)}(A^x_S) \\ 
&=& \sum_{S \subset L_x} \frac{\gamma_{|S|}(t)}{\binom{d}{|S|}} P_{\bf p}(A^x_S). 
\end{eqnarray*}
Hence 
\begin{eqnarray*}
\frac{d}{d t}P_{\gamma(t)}(A)|_{t=0} 
= \sum_{x \in W} \frac{d}{d t}P_{\overrightarrow{\gamma}_x (t)}(A)|_{t=0} 
= \sum_{x \in W} \sum_{S \subset L_x}  \frac{d\gamma_{|S|}}{d t}(0) \frac{P_{\bf p}(A^x_S)}{\binom{d}{|S|}} .  
\end{eqnarray*}
\end{proof}

\begin{corollary}
Let $A$ be an increasing event which depends only on the states of the finite vertices $W \subset V$. 
Then  we have 
\begin{eqnarray*}
\frac{d}{d p} P_{\bf p}(A) 
= \frac{k!(d-k-1)!}{d!}\sum_{x \in W} \sum_{S \subset L_x, |S|=k} \sum_{y \in L_x \backslash S} P_{\bf p} (A^x_{S \cup \{ y \} }\backslash A^x_{S}) 
\end{eqnarray*}
if $p_k=1-p$ and $p_{k+1} =p$.  
\end{corollary}


\section{Number of finite and infinite clusters}\label{sec:Number} 

The {\it number of weak cluster and strong cluster per vertex} are defined as 
\begin{eqnarray*}
\kappa(G,{\bf p}) := \mathbb{E}_{\bf p}(|C|^{-1}) \ \ \text{ and }\ \ \tilde{\kappa}(G,{\bf p}) := \mathbb{E}_{\bf p}(|\tilde{C}|^{-1}),
\end{eqnarray*}
where $|C|^{-1}:=0$ if $|C|=\infty$ and $|\tilde{C}|^{-1}:=0$ if $|\tilde{C}|=\infty$. 
A graph $G$ is called {\it amenable} if there exist non-empty finite subsets $F(n)$ in $V$ satisfying 
$\limsup_{n\to \infty} \frac{|\partial F(n)|}{|F(n)|}=0$, where 
$$\partial F(n) = \{ x \in F(n) \mid \text{there is } y \in V\backslash F(n) \text{ such that } x \sim y \}.$$
Then as in \cite{MR1707339}, we obtain  

\begin{theorem}[\cite{MR1707339}]\label{thm:clusterpervertex}
Assume $G$ is amenable. 
Let $K_n$ $( $and $\tilde{K}_n)$ be the number of weak $($ and strong$)$ cluster in the subgraph $(F(n), E_\omega(F(n)))$ of $G_\omega$, where $E_\omega(F(n)):= \{(x,y) \in E_\omega \mid x,y \in F(n)\}$. 
Then 
\begin{eqnarray*}
\frac{K_n}{|F(n)|} \to \kappa(G,{\bf p}) \ \ \text{ and } \ \ \frac{\tilde{K_n}}{|F(n)|} \to \tilde{\kappa}(G,{\bf p}) 
\end{eqnarray*}
as $n\to \infty$, $P$-a.s.\ and in $L^1(P)$. 
\end{theorem}

For a graph automorphism $f$ and a configuration $\omega$, we define the configuration $f^*\omega \in \Omega $ as $f^*\omega(x):= \omega(f(x))$ for all $x \in V$. 
An event $A$ is called {\it translation invariant} if $f^* A:=\{f^*\omega \mid \omega \in \Omega\} = A$ for any graph automorphism $f$. 
For example, the event $\{\text{there is an infinite cluster}\}$ is translation invariant, but the event $\{ |C(x)|=\infty \}$ is not translation invariant. 
By the Kolmogorov's extension theorem, the probability measure $P_{\bf p}$ is {\it translation invariant}, that is, $P_{\bf p}(f^*A)=P_{\bf p}(A)$ for any event $A$. 
Hence $P_{\bf p}$ is {\it ergodic}, that is,  $P_{\bf p}(A)=0$ or $P_{\bf p}(A)=1$ for any translation invariant event $A$.

\begin{theorem} \label{thm:Numcluster} \ 
\begin{enumerate}
\item Let $N$ be the number of weak infinite cluster. 
If $\theta(G,{\bf p})>0$, then 
$P_{\bf p}(N=1) = 1$ or $P_{\bf p}(N=\infty ) = 1$. 
\item Let $\tilde{N}$ be the number of strong infinite cluster. 
If $\tilde{\theta}(G,{\bf p})>0$, then 
$P_{\bf p}(\tilde{N}=1) = 1$ or $P_{\bf p}(\tilde{N}=\infty ) = 1$. 
\end{enumerate}
\end{theorem}

\begin{proof}
The proof is based on \cite{MR1707339}. 
By Theorem \ref{thm:near0}, we have $p_2+\cdots+p_d>0$. 
Since $N$ is a translation invariant function on $\Omega$, 
there exists $k\in \{0,1,2,\dots\}\cup\{\infty\}$ such that $P_{\bf p}(N=k)=1$. 
Since $\theta(G,{\bf p})>0$, $P_{\bf p}(N=0) \not=1$, that is, $k\in \{1,2,\dots\}\cup\{\infty\}$. 
Set $N_{B(n)}(0)$ and $N_{B(n)}(1)$ to be the maximum and the minimum number of the infinite clusters under admissible changing of the states of vertices in $B(n)$ respectively, where admissible means that we don't consider the configuration $\omega $ containing the state $\omega(x)$ with $|\omega(x)|=k$ if $p_k=0$. 
Since every admissible configuration on $B(n)$ has a strictly positive probability, we have  
\begin{eqnarray*}
P_{\bf p}(N_{B(n)}(0)=N_{B(n)}(1)=k) = 1. 
\end{eqnarray*}
Suppose $P_{\bf p}(N=k)=1$ for some $k \in \{2,3,4,\dots \}$. 
Let $M_{B(n)}$ be the number of the infinite clusters intersecting $B(n)$. 
Then $M_{B(n)}$ is non-decreasing in $n$, and $M_{B(n)} \to N$ as $n\to \infty$. 
Since $k\ge 2$, there is $n\in \mathbb{N}$ such that $P_{\bf p}(M_{B(n)}\ge 2) > 0$. 
If $M_{B(n)}\ge 2$, then there is a path in $B(n)$ from $x$ in an infinite cluster to $y$ in another infinite cluster such that the intermediate vertices are not in infinite clusters. 
If a change of the state of $x$ or of $y$ changes the number $M_{B(n)}$, then $N_{B(n)}(0)>N_{B(n)}(1)$. 
If any change of the state of $x$ and of $y$ does not change the number $M_{B(n)}$, since $p_2+\cdots+p_d>0$, we can connect the infinite clusters containing $x$ and $y$ by changing the states of $x$, $y$ and the intermediate vertices in the path. 
Hence $N_{B(n)}(0)>N_{B(n)}(1)$. 
Thus we have 
\begin{eqnarray*}
P_{\bf p}(N_{B(n)}(0)>N_{B(n)}(1)) > 0  .  
\end{eqnarray*}
This is a contradiction. 
The proof for $\tilde{N}$ is same. 
\end{proof}

\begin{theorem}
\label{thm:weakunique}
Assume $G$ is amenable. 
Let $N$ be the number of infinite weak cluster. 
If $\theta(G,{\bf p})>0$, and $p_0+p_1>0, p_2>0$, then 
\begin{eqnarray*}
P_{\bf p}(N=1) = 1. 
\end{eqnarray*}
\end{theorem}

\begin{proof}
The proof is based on \cite{MR990777}. 
Since $G$ is amenable, there exist non-empty finite subsets $F(n)$ in $V$ satisfying $\limsup_{n\to \infty} \frac{|\partial F(n)|}{|F(n)|}=0$. 
We may choose $F(n)$ to be connected, that is, the subgraph $(F(n),E(F(n)))$ such that $E(F(n)):=\{(x,y)\in E \mid x,y\in F(n)\}$ is connected with respect to $\sim$. 
Let $\partial'F(n):= \{ y \in V \backslash F(n) \mid \text{there is } x \in F(n) \text{ such that } y \sim x \}$. 
Note that  $|F(n)| \to \infty$ as $n\to \infty$.

We say that a vertex $x$ in $V$ is a {\it trifurcation} if 
(1) $x$ is in an infinite weak cluster; 
(2) the deletion of $x$ and the incident edges splits this infinite weak cluster into exactly three disjoint infinite weak clusters. 
Let $T_x:= \{\omega \in \Omega \mid x \text{ is a trifurcation}\}$. 

Since $P_{\bf p}(T_x) = P_{\bf p}(T_y)$ for any $x,y \in V$, we have 
\begin{eqnarray*}
\mathbb{E}_{\bf p}\left(\sum_{x \in F(n)} I_{T_x}\right) 
=\sum_{x \in F(n)} \mathbb{E}_{\bf p}(I_{T_x}) 
= \sum_{x \in F(n)}P_{\bf p}(T_x) 
= |F(n)| P_{\bf p}(T_o).  
\end{eqnarray*}
On the other hand, as the case of the hypercube lattice, we have 
\begin{lemma}[\cite{MR1707339}]
\begin{eqnarray*}
|\{x \in K \cap F(n) \mid x \text{ is a trifurcation} \}| \le |K \cap \partial' F(n)|   
\end{eqnarray*}
for any weak cluster $K$. 
\end{lemma}
Hence 
\begin{eqnarray*}
\sum_{x \in F(n)}I_{T_x} 
= \sum_{K} \sum_{x \in F(n)\cap K}I_{T_x} 
\le \sum_{K} |K \cap \partial' F(n)|
= |\partial' F(n)| 
\end{eqnarray*}
where $K$ runs over all weak clusters in $G_\omega$. 
Hence for any $x \in V$
\begin{eqnarray*}
P_{\bf p}(T_x) 
\le \frac{|\partial' F(n)|}{|F(n)|} 
\le \frac{d|\partial F(n)|}{|F(n)|}
\to 0 \ \ \text{ as } n\to \infty, 
\end{eqnarray*}
that is, $P_{\bf p}(T_x) =0$. 

By Theorem \ref{thm:Numcluster}, we have $P_{\bf p}(N=1) = 1$ or $P_{\bf p}(N=\infty) = 1$.  
Suppose $P_{\bf p}(N=\infty) = 1$. 
Let $M_{F(n)}$ be the number of the infinite weak clusters intersecting $F(n)$, and $M_{F(n)}(0)$ the maximum number of the infinite weak clusters intersecting $F(n)$ under admissible changing of the states of vertices in $F(n)$. 
Since 
\begin{eqnarray*}
P_{\bf p}(M_{F(n)}\ge 3) \to P_{\bf p}(N \ge 3)=1 \ \text{ as } n \to \infty 
\end{eqnarray*}
and $M_{F(n)}(0) \ge M_{F(n)}$, there exists $n\in \mathbb{N}$ such that 
\begin{eqnarray*}
P_{\bf p}(M_{F(n)}(0)\ge 3) > 0.   
\end{eqnarray*}
For $x,y,z \in \partial' F(n)$ let $L_{F(n)}(x,y,z)$ be the event that $x,y,z$ are weakly connected to disjoint infinite weak clusters in $(G_\omega \backslash (F(n) \cup \partial' F(n)))\cup \{x,y,z \}$ respectively. 
Then we have 
\begin{eqnarray*}
P_{\bf p}(M_{F(n)}(0)\ge 3) 
&=& P_{\bf p}\left( \left( \cup_{x,y,z \in \partial' F(n) } L_{F(n)}(x,y,z) \right) \cap \{M_{F(n)}(0)\ge 3\} \right) \\ 
&\le & \sum_{x,y,z \in \partial' F(n) } P_{\bf p}\left( L_{F(n)}(x,y,z)  \cap \{M_{F(n)}(0)\ge 3\} \right) . 
\end{eqnarray*}
Thus there is $x,y,z \in \partial' F(n) $ such that 
$$P_{\bf p}\left( L_{F(n)}(x,y,z) \cap \{M_{F(n)}(0)\ge 3\} \right) >0.$$

Since $F(n)$ is connected, there is a self avoiding path $\gamma $ from $x$ to $y$ whose intermediate vertices are in $F(n)$, and there is a self avoiding path $\gamma '$ from $z$ to a vertex $v$ in $\gamma \cap F(n)$ whose intermediate vertices are also in $F(n)$.  
Then there exists a spanning tree $T$ of $F'(n):=F(n) \cup \partial' F(n)$ containing $\gamma $ and $\gamma '$ such that $(u,w) \in T$ if and only if $(w,u) \in T$, and if $(u,w) \in T$ and $w\in \partial' F(n)$ then $u\in F(n)$. 
Let $\partial'' F(n) := \{ u \in V \backslash F'(n) \mid \text{there is } w \in F'(n) \text{ such that } u \sim w \}$,  
and $L'(n,x,y,z) := L_{F(n)}(x,y,z) \cap \{M_{F(n)}(0) \ge 3\}$. 

For each configuration $\omega \in L'(n,x,y,z)$, we change the states on $F'(n)$ such that $v$ is a trifurcation as follows:  
Let $\omega'$ be the configuration such that 
\begin{enumerate}
\item for $w \in V \backslash F'(n)$ and $w\in \{x,y,z\}$, $\omega'(w):=\omega(w)$;  
\item for $w $ on $\gamma $ and $\gamma' $ except $v,x,y,z$,  $\omega'(w) :=\{w(i),w(j)\}$ such that $(w,w(i)), (w,w(j))\in T$;
\item $\omega'(v) :=\{v(i),v(j)\}$ where $i,j$ are the first and second smallest numbers such that $v(i)$, $v(j)$ are vertices on $\gamma $ or $\gamma' $;
\item for $w \in F'(n)$ except vertices on $\gamma $ and $\gamma' $,
\begin{enumerate}
\item if $p_0>0$, $\omega(w)=\emptyset$;   
\item if $p_0=0$,   
\begin{enumerate}
\item if there is no $w(j) \in \partial'' F(n)$ with $w \in \omega(w(j)) $,  
then $\omega'(w) :=\{w(i)\}$ such that $i$ is the minimum number with $(w,w(i)) \in T$ and $w(i) \in F(n)$;
\item if there is $w(j) \in \partial'' F(n)$ with $w \in \omega(w(j)) $,  
then $\omega'(w) :=\{w(i)\}$, where $i$ is the minimum number with $w(i) \in \partial'' F(n)$ and $w \in \omega(w(i)) $. 
\end{enumerate}
\end{enumerate}
\end{enumerate}
Then $v$ is a trifurcation. 
We define $J(n):=\{ \omega' \mid \omega \in L'(n,x,y,z)\}$,  
then  $ J(n) \subset L'(n,x,y,z)\cap T_v$.   
Hence we have 
\begin{eqnarray*}
P_{\bf p}(T_v) 
& \ge & P_{\bf p}(J(n)) \\ 
& \ge & P_{\bf p}(J(n) \mid L'(n,x,y,z)) P_p(L'(n,x,y,z)) \\
& \ge & \min \left\{ \max\left\{ p_0, \frac{p_1}{d} \right\} ,\frac{2}{d(d-1)}p_2 \right\}^{|F'(n)|} 
P_{\bf p}(L'(n,x,y,z)) \\
& > & 0. 
\end{eqnarray*}
This is a contradiction. 
\end{proof}

\begin{theorem}
\label{thm:strongunique}
Assume $G$ is amenable. 
Let $\tilde{N}$ be the number of infinite strong cluster. 
If $\tilde{\theta}(G,{\bf p})>0$, and $p_0+p_1>0$, $p_2>0$, $p_3+\cdots + p_d>0$,then 
\begin{eqnarray*}
P_{\bf p}(\tilde{N}=1) = 1. 
\end{eqnarray*}
\end{theorem}

\begin{proof}
As Theorem \ref{thm:weakunique} we take a sequence of vertex sets $F(n)$. 
Let $\tilde{M}_{F(n)}(0)$ be the maximum number of the infinite strong clusters intersecting $F(n)$ under admissible  changing of the states of vertices in $F(n)$, 
and $\tilde{T}_x:= \{\omega \in \Omega \mid x$ is a trifurcation in the sense of strong cluster$\}$. 
For $x,y,z \in \partial' F(n)$ let $\tilde{L}_{F(n)}(x,y,z)$ be the event that $x,y,z$ are strongly connected to disjoint infinite strong clusters in $(G_\omega \backslash (F(n) \cup \partial' F(n)))\cup\{x,y,z\}$ respectively.
Denote $\tilde{L}(n,x,y,z) := \tilde{L}_{F(n)}(x,y,z) \cap \{\tilde{M}_{F(n)}(0) \ge 3\}$. 
 
As the proof of Theorem \ref{thm:weakunique}, it is sufficient to define $ \omega ' $ for $\omega \in \tilde{L}(n,x,y,z)$ 
such that $\tilde{J}(n):=\{ \omega' \mid \omega \in \tilde{L}(n,x,y,z)\}$ is a subset of $\tilde{L}(n,x,y,z) \cap \tilde{T}_v $ and  
$P_{\bf p}(\tilde{J}(n) \mid \tilde{L}(n,x,y,z))>0$ for sufficiently large $n$ and any $x,y,z \in \partial' F(n) $. 
Let $T$ be a spanning tree of $F'(n)$ as in the proof of Theorem \ref{thm:weakunique}. 
Let $\omega'$ be the configuration such that 
\begin{enumerate}
\item for $w \in V \backslash F'(n)$, $\omega'(w):=\omega(w)$;  
\item for $w $ on $\gamma $ and $\gamma' $ except $x,y,z,v$, $\omega'(w) :=\{w(i),w(j)\}$ such that $(w,w(i)), (w,w(j)) \in T$;
\item $\omega'(v) := \{v(i),v(j),v(k), v(h_4), \dots, v(h_l) \}$ where $v(i)$, $v(j)$, $v(k)$ are vertices on $\gamma $ and $\gamma' $, and $h_4, \dots, h_l$ are the small numbers in $\{ 1,2,\dots,d\} \backslash \{i,j,k\}$ with $l=\min \{m \ge 3 \mid p_m>0\} $;
\item for $w\in \{x,y,z\}$ 
\begin{enumerate}
\item if $ \omega(w)$ contains the vertex in $\gamma $ or $\gamma'$ next to $w$, then $\omega'(w):=\omega(w)$; 
\item if $ \omega(w)$ does not contain the vertex in $\gamma $ or $\gamma'$ next to $w$, 
then $\omega'(w) := \{w(i),w(j)\}$, where $w(i)$ is in $\gamma $ or $\gamma'$, and $j$ is the smallest number such that $w(j)$ in an infinite  strong cluster of $G_\omega \backslash (F(n) \cup \partial' F(n))$, and $w_j \in \omega(w)$. 
\end{enumerate}
\item for $w \in F'(n)$ except vertices on $\gamma $ and $\gamma' $,
\begin{enumerate}
\item if $p_0>0$, $\omega'(w)=\emptyset$;   
\item if $p_0=0$, $\omega'(w) :=\{w(i)\}$ such that $(w,w(i)) \in T$ and $w(i) \in F(n)$. 
\end{enumerate}
\end{enumerate}
Then $v$ is a trifurcation and we get $\tilde{J}(n) \subset \tilde{L}(n,x,y,z) \cap \tilde{T}_v $,    
and we obtain 
\begin{eqnarray*}
P_{\bf p}(\tilde{J}(n) \mid \tilde{L}(n,x,y,z)) 
\ge \min \left\{ \frac{h!(d-h)!}{d!} p_h \mid p_h>0 \right\}^{|F'(n)|} 
> 0. 
\end{eqnarray*}
\end{proof}

\section{Exponential decay}\label{sec:expdecay}

Let 
\begin{eqnarray*}
c({\bf p},d) 
:= \sum_{i=1}^d p_i \frac{i}{d} \ \ \ \text{ and } \ \ \ 
\tilde{c}({\bf p},d) 
:= \frac{1}{d^2}\left( \sum_{i=1}^d p_i \frac{i}{d} \right) \left(\sum_{i=2}^d p_i \frac{i(i-1)}{d(d-1)} \right).
\end{eqnarray*} 
Then $c({\bf p},d)=P_{\bf p}( y \in \omega(x)) $ and $\tilde{c}({\bf p},d)=P_{\bf p}( y \in \omega(x))P_{\bf p}( \{u,w\} \subset \omega(v))/d^2 $ 
for any $u,v,w,x,y \in V$ with $x\sim y$, $v \sim u$, $v \sim w$.

\begin{theorem}\label{thm:weakBoundary}
Suppose $\chi(G,{\bf p}) < \infty$. 
Then there exists $\alpha(G,{\bf p})>0$ such that 
\begin{eqnarray*}
P_{\bf p}(o \leftrightharpoons \partial B(n)) \le 2 e^{-n \alpha(G,{\bf p})} 
\end{eqnarray*}
for any sufficiently large $n$.
\end{theorem}

\begin{proof}
The proof is based on \cite{MR1707339}. 
We may assume $p_0<1$. 
Let $N_n := |\{x\in \partial B(n) \mid o \leftrightharpoons x\}| $ and 
$\tau_{\bf p}(o,x):=P_{\bf p}(o \leftrightharpoons x)$ for $x\in V$. 
Then we have 
\begin{eqnarray*}
\sum_{n=0}^\infty \mathbb{E}_{\bf p}(N_n) 
&=& \sum_{n=0}^\infty \sum_{x\in \partial B(n)} \tau_{\bf p}(o,x) 
= \sum_{x\in V} \tau_{\bf p}(o,x) \\ 
&=& \mathbb{E}_{\bf p}(|\{x\in V \mid o \leftrightharpoons x\}|) 
= \chi(G,{\bf p})
< \infty. 
\end{eqnarray*}
Hence $\mathbb{E}_{\bf p}(N_n) \to 0$ as $n \to \infty$, and 
there exists $m_0\in \mathbb{N}$ such that $\mathbb{E}_{\bf p}(N_{m_0})<\frac{c({\bf p},d)}{2}$. 

Let $A_x :=\{ \exists y\in L_x \text{ s.t.\ } y \leftrightharpoons \partial B(k)\}$ and $B_x:= \{ x \leftrightharpoons \partial B(k)\}$. 
Note that $A_x$ is independent of the state at $x$. 
Let $A_{x,i}  \subset A_x$ be the event that $i$ is the minimum number such that $x(i) \leftrightharpoons \partial B(k)$ in $V \backslash \{x\}$. 
The family $\{A_{x,i}\}_i$ is a partition of $A_x$. 
For $\omega \in A_{x,i}$ and $ S \subset L_x $ with $x(i)\in S$, we set 
\begin{eqnarray*}
\omega_S'(z) =
\begin{cases}
\omega(z) & \text{ if } z \not =x \\ 
S & \text{ if } z = x .  
\end{cases}
\end{eqnarray*}
and 
\begin{eqnarray*}
A_{x,i}' := \bigcup_{\omega\in A_{x,i}} \bigcup_{S \subset L_x ;  x(i) \in S}\{\omega_S'\} \subset A_{x,i} \cap B_x.
\end{eqnarray*} 
Since $\{A_{x,i}\}_i$ is a partition of $A_x$, $A_{x,i}'$ are disjoint each other in  $i$. 
Hence 
\begin{eqnarray*}
P_{\bf p}(B_x) 
&\ge& P_{\bf p}(\sqcup_{i=1,\dots,d}A_{x,i}') 
= \sum_{i=1,\dots,d} P_{\bf p}(A_{x,i}') \\
&=& \sum_{i=1,\dots,d} P_{\bf p}(A_{x,i}'|A_{x,i}) P_{\bf p}(A_{x,i}) \\ 
&=& \sum_{i=1,\dots,d} P_{\bf p}(x(i) \in \omega(x))  P_{\bf p}(A_{x,i})  \\ 
&=& c({\bf p}, d) P_{\bf p}(\sqcup_{i=1,\dots,d} A_{x,i}) \\ 
&=& c({\bf p}, d) P_{\bf p}(A_x) .  
\end{eqnarray*}
Using Reimer's inequality, and by the vertex transitivity, we have 
\begin{eqnarray*}
&& P_{\bf p}(o \leftrightharpoons \partial B(m+k)) \\ 
&\le & P_{\bf p} \left(\bigcup_{x\in \partial B(m)} \{o \leftrightharpoons x\} \Box \{\exists y\in L_x \text{ s.t.\ } y \leftrightharpoons \partial B(m+k)\} \right) \\
&\le & \sum_{x\in \partial B(m)} P_{\bf p} (o \leftrightharpoons x) P_{\bf p}\left(\exists y\in L_x \text{ s.t.\ } y \leftrightharpoons \partial B(m+k) \right) \\
&\le & \sum_{x\in \partial B(m)} \tau_{\bf p}(o,x) c({\bf p},d)^{-1} P_p\left(x \leftrightharpoons \partial B(m+k) \right) \\
&\le &  c({\bf p},d)^{-1} \sum_{x\in \partial B(m)} \tau_{\bf p}(o,x)P_{\bf p} \left(x \leftrightharpoons \partial B(x,k) \right) \\
&=&  c({\bf p},d)^{-1} \sum_{x\in \partial B(m)} \tau_{\bf p}(o,x) P_{\bf p} \left(o \leftrightharpoons \partial B(k) \right) \\
&\le& c({\bf p},d)^{-1} P_{\bf p} \left(o \leftrightharpoons \partial B(k) \right) \mathbb{E}_{\bf p}(N_m) 
\end{eqnarray*}
where $\partial B(x,k)=\{v\in V \mid \delta(x,v) = k\} $. 
Since for any $n$ there are $r,s\in \mathbb{N}$ such that $n=m_0r+s$ and $0\le s < m_0$, 
we have 
\begin{eqnarray*}
P_{\bf p}(o \leftrightharpoons \partial B(n)) 
&\le & P_{\bf p}(o \leftrightharpoons \partial B(m_0r)) \\ 
&\le & c({\bf p},d)^{-1} \mathbb{E}_{\bf p}(N_{m_0}) P_{\bf p}(o \leftrightharpoons \partial B(m_0r-m_0)) \\ 
& \le &  \dots \\
&\le & (c({\bf p},d)^{-1} \mathbb{E}_{\bf p}(N_{m_0}))^{r-1} P_{\bf p}(0 \leftrightharpoons \partial B(m_0)) \\ 
&\le & (c({\bf p},d)^{-1} \mathbb{E}_{\bf p}(N_{m_0}))^r\\ 
&\le & 2^{-r} = 2^{\frac{s}{m_0}} 2^{-\frac{n}{m_0}} \\
&\le & 2 e^{-n\alpha(G,{\bf p})}
\end{eqnarray*}
where $\alpha(G,{\bf p}) := \frac{\log 2}{m_0}>0$. 
\end{proof}

\begin{theorem}\label{thm:strongBoundary}
Suppose $\tilde{\chi}(G,{\bf p}) < \infty$. 
Then there exists $\tilde{\alpha}(G,{\bf p})>0$ such that 
\begin{eqnarray*}
P_{\bf p}(o \leftrightarrow \partial B(n)) \le 2 e^{-n \tilde{\alpha}(G,{\bf p})} 
\end{eqnarray*}
for any sufficiently large $n$.
\end{theorem}

\begin{proof}
The proof is also based on \cite{MR1707339}. 
We may assume $p_0+p_1<1$. 
For $k\ge 3$ and $x\in B(k-2)$, let $A_x :=\{ \exists y\in L_x \text{ s.t.\ } y \leftrightarrow \partial B(k)\}$ and $B_x:= \{ x \leftrightarrow \partial B(k)\}$. 
For $\partial B(x,2)$, we fix the order of the vertices as $\partial B(x,2)=\{z_1,z_2, \dots ,z_m\}$ where $m=|\partial B(x,2)|$. 
Let $A_{x,l} \subset A_x$ be the event that $l$ is the minimum number such that $z_l \leftrightarrow \partial B(k)$ in $V \backslash B(x,1)$ and $\omega(z_l) \cap L_x \not=\emptyset$. 

For each $\omega \in A_{x,l}$, let $j=j(\omega)\in \{1,2,\dots,d \}$ be the minimum number such that $x(j) \in \omega(z_l)$. 
For $\omega \in A_{x,l}$, $ S \subset L_x $ with $ x(j) \in S$, and $ R \subset L_{x(j)} $ with $ \{x, z_l\} \subset R$, we set 
\begin{eqnarray*}
\omega_{S,R}'(v) =
\begin{cases}
\omega(v) & \text{ if } v \not =x,x(j) \\ 
S & \text{ if } v = x \\ 
R & \text{ if } v = x(j). 
\end{cases}
\end{eqnarray*}
Let 
\begin{eqnarray*}
A_{x,l}' := \bigcup_{\omega\in A_{x,l}} \bigcup_{S \subset L_x ; x(j) \in S} \bigcup_{R\subset L_{x(j)}; \{x, z_l\} \subset R}\{\omega_{S,R}'\} .
\end{eqnarray*} 
Since $A_{x,l}$ is independent of the states on $B(x,1)$, we have $ A_{x,l}' \subset A_{x,l} \cap B_x $. 
Hence 
\begin{eqnarray*}
P_{\bf p}(B_x) 
&\ge & P_{\bf p}(A'_{x,l}) \\ 
&=& P_{\bf p}( A'_{x,l} | A_{x,l}) P_{\bf p}( A_{x,l})\\ 
&=& P_{\bf p}(x(j) \in \omega(x)) P_{\bf p}(\{x, z_l \} \subset \omega(x_j)) P_{\bf p}(A_{x,l})  \\ 
&=& d^2\tilde{c}({\bf p}, d) P_{\bf p}(A_{x,l}) . 
\end{eqnarray*}
Thus we have 
\begin{eqnarray*}
P_{\bf p}(A_x) 
&=& P_{\bf p}(\cup_{l=1}^m A_{x,l}) \\
&\le & \sum_{l=1}^m P_{\bf p}( A_{x,l}) \\
&\le & \sum_{l=1}^m d^{-2}\tilde{c}({\bf p}, d)^{-1}  P_{\bf p}( B_x) \\
&\le & \tilde{c}({\bf p}, d)^{-1}  P_{\bf p}( B_x) . 
\end{eqnarray*}
The rest of the proof is same to the case of the weak connection. 
\end{proof}

\begin{theorem}\label{thm:NumVertex}
\begin{enumerate}
\item Suppose $\chi(G,{\bf p}) < \infty$. 
Then we have 
\begin{eqnarray*}
P_{\bf p}(|C|\ge n ) 
\le  e^{-n\frac{c({\bf p},d)^2}{2 \chi(G,{\bf p})^2}} 
\end{eqnarray*}
for $n > \frac{\chi(G,{\bf p})^2}{c({\bf p},d)^2}$.
\item Suppose $\tilde{\chi}(G,{\bf p}) < \infty$. 
Then we have 
\begin{eqnarray*}
P_{\bf p}(|\tilde{C}|\ge n ) 
\le  e^{-n\frac{\tilde{c}({\bf p},d)^2}{2 \tilde{\chi}(G,{\bf p})^2}} 
\end{eqnarray*}
for $n > \frac{\tilde{\chi}(G,{\bf p})^2}{\tilde{c}({\bf p},d)^2}$.
\end{enumerate}
\end{theorem}

\begin{proof}
The proof is based on \cite{MR762034} (cf.\ \cite{MR1707339}). 
A {\it skeleton} is a  (non-directed) tree whose vertices have degree 1 ({\it exterior vertices}) or degree 3 ({\it interior vertices}). 
A skeleton with $k$ exterior vertices has $k-2$ interior vertices, and $2k-3$ edges.  
A skeleton with $k$ exterior vertices is called {\it labelled} 
if there exists an assignment of the numbers $0,1,2, \dots ,k-1$ to the exterior vertices. 
Two labelled skeletons are called {\it isomorphic} if there exists a one to one correspondence between their vertex sets under which both the adjacency relation and the labellings of the exterior vertices are preserved. 

\begin{lemma}[\cite{MR1707339}] 
Let $N_{n+1}$ be the number of labelled skeletons with $n+1$ exterior vertices. 
Then we have 
\begin{eqnarray*}
N_{n+1} = \frac{(2n-2)!}{2^{n-1}(n-1)!}
\end{eqnarray*}
\end{lemma}

\begin{lemma}[\cite{MR1707339}] 
Let $G=(V,E)$ be a non-directed connected graph. 
For any $X=(x_0,x_1, \dots ,x_k) \in V^{k+1}$,  
there is a labelled skeleton $S$ with $k+1$ exterior vertices together with a mapping $\psi_X$ from the vertex set of $S$ into $V$ 
such that 
{\rm (a)} the exterior vertices of $S$ with label $i$ is mapped to $x_i$ by $\psi_X$ for $i=0,1,\dots,k$, and 
{\rm (b)} the edges of $S$ corresponds to paths joining the $2k-1$ pairs which are disjoint except the end vertices, which may be a trivial path. 
\end{lemma}

Let $X=(x_0,x_1,\dots,x_n) \in V^{n+1}$ with $x_0=o$, and $S$ be a labelled skeleton with $n+1$ exterior vertices. 
We call a map $\psi_X:S \to G$ in the above lemma an admissible mapping. 
First if $s \in V(S)$ is a farthest vertex from the labeled vertex $0$ with respect to the graph distance of $S$, then there is only one adjacent vertex $t$ with $\delta(t,0)= \delta(s,0) -1$. 
If $\psi_X(s) \not=\psi_X(t)$, then using Reimer's inequality, we have  
\begin{eqnarray*}
&& P_{\bf p}\left( \psi_X(u) \leftrightharpoons \psi_X(v) \text{ for all } \{u,v\} \in E(S) \right) \\
&\le & P_{\bf p}\left( (\psi_X(u) \leftrightharpoons \psi_X(v) \text{ for all } \{u,v\} \in E(S)\backslash \{\{s,t\}\}) \Box (\cup_{v \in L_{\psi_X(t)}} \{ v \leftrightharpoons \psi_X(s) \} ) \right) \\
&\le & P_{\bf p}\left( \psi_X(u) \leftrightharpoons \psi_X(v) \text{ for all } \{u,v\} \in E(S)\backslash \{\{s,t\}\}\right) 
P_{\bf p}\left(\cup_{v \in L_{\psi_X(t)}} \{ v \leftrightharpoons \psi_X(s) \}   \right) \\
&\le & P_{\bf p}\left( \psi_X(u) \leftrightharpoons \psi_X(v) \text{ for all } \{u,v\} \in E(S)\backslash \{\{s,t\}\}\right) 
c({\bf p},d)^{-1} P_{\bf p}\left(\psi_X(t) \leftrightharpoons \psi_X(s) \right) . 
\end{eqnarray*}
The last inequality is derived as in the proof of Theorem \ref{thm:weakBoundary}. 
If $\psi_X(s) = \psi_X(t)$, then since $P_{\bf p}\left(\psi_X(t) \leftrightharpoons \psi_X(s) \right)=1$ and $c({\bf p},d)^{-1}\ge 1$, we have  
\begin{eqnarray*}
&& P_{\bf p}\left( \psi_X(u) \leftrightharpoons \psi_X(v) \text{ for all } \{u,v\} \in E(S) \right) \\
&\le & P_{\bf p}\left( \psi_X(u) \leftrightharpoons \psi_X(v) \text{ for all } \{u,v\} \in E(S)\backslash \{\{s,t\}\}\right) 
c({\bf p},d)^{-1} P_{\bf p}\left(\psi_X(t) \leftrightharpoons \psi_X(s)  \right) . 
\end{eqnarray*}
Next for $S\backslash \{s\}:=(V(S) \backslash \{s\}, E(S)\backslash \{s,t\})$, let $s' \in V(S)  \backslash \{s\}$ be a farthest vertex from the labeled vertex $0$ of $S\backslash \{s\}$, then there is only one adjacent vertex $t'$ with $\delta(t',0)= \delta(s',0) -1$. 
By the same procedure, we obtain 
\begin{eqnarray*}
&& P_{\bf p}\left( \psi_X(u) \leftrightharpoons \psi_X(v) \text{ for all } \{u,v\} \in E(S) \right) \\
&\le & P_{\bf p}\left( \psi_X(u) \leftrightharpoons \psi_X(v) \text{ for all } \{u,v\} \in E(S)\backslash \{ \{s,t\} ,\{s',t'\}\}\right) \\ 
&& \ c({\bf p},d)^{-1}P_{\bf p}\left(\psi_X(t') \leftrightharpoons \psi_X(s')  \right)  
c({\bf p},d)^{-1} P_{\bf p}\left(\psi_X(t) \leftrightharpoons \psi_X(s) \right) . 
\end{eqnarray*}
Inductively, we get 
\begin{eqnarray*}
&& P_{\bf p}\left( \psi_X(u) \leftrightharpoons \psi_X(v) \text{ for all } \{u,v\} \in E(S) \right) \\
&\le & P_{\bf p}\left( \psi_X(0)  \leftrightharpoons \psi_X(r) \text{ for } \{0,r\} \in E(S) \right) 
\prod_{\{t,s\} \in E(S), t,s\not=0} \left( c({\bf p},d)^{-1} P_{\bf p}\left(\psi_X(t) \leftrightharpoons \psi_X(s)  \right) \right) \\  
&\le & \prod_{\{t,s\} \in E(S)} \left( c({\bf p},d)^{-1} P_{\bf p}\left(\psi_X(t) \leftrightharpoons \psi_X(s)  \right) \right) \\ 
&\le & c({\bf p},d)^{-|E(S)|} \prod_{\{t,s\} \in E(S)} P_{\bf p}\left(\psi_X(t) \leftrightharpoons \psi_X(s)  \right) . 
\end{eqnarray*}
Set $\tau_{\bf p}(x_0,x_1,\dots,x_n):=P_{\bf p}(x_0 \leftrightharpoons x_1  \leftrightharpoons  \dots  \leftrightharpoons x_n \text{ in }C)$. 
Then we have 
\begin{eqnarray*}
\tau_{\bf p}(x_0,x_1,\dots,x_n) 
&=& P_{\bf p} \left( \bigcup_{S}\bigcup_{\psi_X:S \to G}\left( \psi_X(u) \leftrightharpoons \psi_X(v) \text{ for all } \{u,v\} \in E(S) \right)\right) \\ 
&\le & \sum_{S}\sum_{\psi_X:S \to G} P_{\bf p}\left( \psi_X(u) \leftrightharpoons \psi_X(v)  \text{ for all } \{u,v\} \in E(S) \right) \\ 
&\le &  \sum_{S}c({\bf p},d)^{-|E(S)|} \sum_{\psi_X:S \to G} \prod_{\{t,s\} \in E(S)} P_{\bf p}\left(\psi_X(t) \leftrightharpoons \psi_X(s)  \right)
\end{eqnarray*}
where $S$ runs over all labelled skeletons with $k+1$ exterior vertices, and $\psi_X:S \to G$ runs over all admissible mappings. 
Since $|C|=\sum_{x\in V} I_{\{o \leftrightharpoons x\}}$, we have 
\begin{eqnarray*}
\mathbb{E}_{\bf p}(|C|^n) 
&=& \mathbb{E}_{\bf p}\left(\left(\sum_{x\in V} I_{\{o \leftrightharpoons x\}}\right)^n\right) \\ 
&=& \mathbb{E}_{\bf p}\left(\sum_{x_1,\dots,x_n\in V} I_{\{o \leftrightharpoons x_1\}}\cdots I_{\{o \leftrightharpoons x_n\}}\right) \\ 
&=& \sum_{x_1,\dots,x_n\in V} \mathbb{E}_{\bf p}\left(I_{\{o \leftrightharpoons x_1\}}\cdots I_{\{o \leftrightharpoons x_n\}}\right) \\ 
&=& \sum_{x_1,\dots,x_n\in V} \tau_{\bf p}(o,x_1,\dots,x_n). 
\end{eqnarray*}
Since $|E(S)|=2n-1$, we get	@
\begin{eqnarray*}
\mathbb{E}_{\bf p}(|C|^n) 
&\le &   \sum_{x_1,\dots,x_n\in V} \sum_{S} c({\bf p},d)^{-|E(S)|} \sum_{\psi_X:S \to G} \prod_{\{t,s\} \in E(S)} P_{\bf p}\left(\psi_X(t) \leftrightharpoons \psi_X(s)  \right) \\ 
&\le & c({\bf p},d)^{-2n+1} \sum_{S}\sum_{\psi:V(S) \to V, \psi(0)=o} \prod_{\{t,s\} \in E(S)} P_{\bf p}\left(\psi(t) \leftrightharpoons \psi(s)  \right) \\
&\le & c({\bf p},d)^{-2n+1}  \sum_{S} \sum_{(y_1,y_2, \dots, y_{2n-1}) \in V^{2n-1}} \prod_{i=1}^{2n-1} P_{\bf p}\left(o \leftrightharpoons y_i  \right) \\
&=& c({\bf p},d)^{-2n+1}  \sum_{S} \left( \sum_{y \in V} P_{\bf p}\left( o \leftrightharpoons y   \right)\right)^{2n-1} \\ 
&\le & N_{n+1}\left( \frac{\chi(G,{\bf p})}{c({\bf p},d)} \right)^{2n-1} 
= \frac{(2n-2)!}{2^{n-1}(n-1)!}\left(\frac{\chi(G,{\bf p})}{c({\bf p},d)}\right)^{2n-1} 
\end{eqnarray*}
where maps $\psi:V(S) \to V$ runs over all maps with $\psi(0)=o$. 
Thus 
\begin{eqnarray*}
\mathbb{E}_{\bf p}(|C|e^{t|C|}) 
&=& \sum_{m=0}^\infty  \frac{t^m}{m!}  \mathbb{E}_{\bf p}(|C|^{m+1}) \\ 
&\le & \sum_{m=0}^\infty  \frac{t^m}{m!}  \frac{(2m)!}{2^{m}m!}\left( \frac{\chi(G,{\bf p})}{c({\bf p},d)}\right)^{2m+1} \\ 
&\le & \frac{\chi(G,{\bf p})}{c({\bf p},d)}  \sum_{m=0}^\infty \binom{2m}{m} \left( \frac{t}{2}  \frac{\chi(G,{\bf p})^2}{c({\bf p},d)^2} \right)^m . 
\end{eqnarray*}
For $t:= \frac{c({\bf p},d)^2}{2 \chi(G,{\bf p})^2 }-\frac{1}{2n}$, from the assumption $n > \frac{\chi(G,{\bf p})^2}{c({\bf p},d)^2} $ we have  
\begin{eqnarray*}
0< \frac{t}{2}  \frac{\chi(G,{\bf p})^2}{c({\bf p},d)^2} 
= \frac{1}{4}\left( 1- \frac{\chi(G,{\bf p})^2 }{n c({\bf p},d)^2}   \right) < \frac{1}{4}.  
\end{eqnarray*}
Thus we have 
\begin{eqnarray*}
\mathbb{E}_{\bf p}(|C|e^{t|C|}) 
\le \frac{\chi(G,{\bf p})}{c({\bf p},d)}  \sum_{m=0}^\infty \binom{2m}{m} \left( \frac{t}{2}  \frac{\chi(G,{\bf p})^2}{c({\bf p},d)^2}  \right)^m
= \frac{\chi(G,{\bf p})}{c({\bf p},d)}  \frac{1}{\sqrt{1-2t \frac{\chi(G,{\bf p})^2}{c({\bf p},d)^2}   }} 
=  \sqrt{n}. 
\end{eqnarray*}
By Markov's inequality, we have  
\begin{eqnarray*}
P_{\bf p}(|C|\ge n ) 
= P_{\bf p}(|C|e^{t|C|} \ge n e^{tn}) 
\le  \frac{\mathbb{E}_{\bf p}(|C|e^{t|C|})}{n e^{tn}} 
\le  \frac{\sqrt{n}}{n e^{n\frac{c({\bf p},d)^2}{2 \chi(G,{\bf p})^2}-\frac{1}{2}}} 
\le  e^{-n\frac{c({\bf p},d)^2}{2 \chi(G,{\bf p})^2}} 
\end{eqnarray*}
for $n > \frac{\chi(G,{\bf p})^2}{c({\bf p},d)^2}$. 
The proof for the strong connection case is same. 
\end{proof}

The hypercube lattice $\mathbb{L}^D$ is the directed graph $(V,E)$ with $V=\mathbb{Z}^D$ and $(x,y)\in E $ if and only if $\|x-y\|=1$, where $\|\cdot \|$ is the Euclidean norm. 
This is a $d$-regular graph with $d=2D$, and  $\mathbb{L}^1=T_2$.

\begin{theorem}
\label{thm:weaklog}
Suppose $p_2>0$ and $D\ge 2$. 
Then there exists the limit 
\begin{eqnarray*}
\zeta(\mathbb{L}^D,{\bf p}):= \lim_{n\to \infty} \left\{ \frac{-\log P_{\bf p}(|C|=n)}{n} \right\} \ge 0
\end{eqnarray*}
and if $\chi(\mathbb{L}^D,{\bf p}) <\infty$, then $\zeta(\mathbb{L}^D,{\bf p}) >0$. 
Moreover,  
\begin{eqnarray*}
P_{\bf p}(|C|=n) \le \frac{n}{a(d,{\bf p})} e^{-n\zeta(\mathbb{L}^D,{\bf p})}
\end{eqnarray*}
for any $n\in \mathbb{N}$, where 
\begin{eqnarray*}
a(d,{\bf p})=\left(\frac{2p_2}{d(d-1)}\right)^3\left( \sum_{i=0}^{d-2} p_i \frac{(d-i)(d-i-1)}{d(d-1)}\right)^{2d-4}\left( \sum_{i=0}^{d-1} p_i \frac{d-i}{d}\right)^{d-4}.
\end{eqnarray*} 
\end{theorem}

\begin{proof}
The proof is based on \cite{MR0496290} (cf.\ \cite{MR1707339}). 
Let $\pi_n:=P_{\bf p}(|C|=n)$. 
\begin{lemma}\label{lem:Fekete3}
We have 
\begin{eqnarray*}
\frac{\pi_{m+n+3}}{m+n+3} \ge a(d,{\bf p}) \frac{\pi_{m}}{m}\frac{\pi_{n}}{n} 
\end{eqnarray*}
for all $n,m \ge 1$. 
\end{lemma}

\begin{proof}
For any subgraph $S$ of $\mathbb{L}^D$, 
set $$\tr(S):=(x^1,x^2,\dots,x^D) \in \mathbb{Z}^D$$ if $x^1 = \max_{ (y^1,y^2, \dots, y^D)\in S}y^1$, $x^2 = \max_{(x^1,y^2, \dots, y^D)\in S} y^2 $, $\dots$, $x^D = \max_{(x^1,x^2, \dots, x^{D-1}, y^D)\in S}y^D $, and  
$$\bl(S):=(x^1,x^2,\dots,x^D) \in \mathbb{Z}^D$$ if $x^1 = \min_{ (y^1,y^2, \dots, y^D)\in S}y^1$, $x^2 = \min_{(x^1,y^2, \dots, y^D)\in S} y^2 $, $\dots$, $x^D = \min_{(x^1,x^2, \dots, x^{D-1}, y^D)\in S}y^D $. 

Set $e_i:=(0, \dots, 0,1,0,\dots,0)$ to be the standard basis of $\mathbb{R}^D$ for $i=1,2,\dots, D$.  
Let $\sigma , \tau$ be connected subgraphs of $\mathbb{L}^D$ with $\tr(\sigma)=-2e_1$, $\bl(\tau)=2e_1$ and $|\sigma|=m$, $|\tau|=n$ for $m,n \in \mathbb{N}$.  
We define the subgraph $\sigma * \tau$ such that 
\begin{eqnarray*}
V(\sigma * \tau) &:=& V(\sigma) \cup V(\tau) \cup \{ -e_1, 0, e_1\} \\ 
E(\sigma * \tau) &:=& E(\sigma) \cup E(\tau ) \cup \{ (-e_1, -2e_1), (-e_1, 0), (0,-e_1), (0,e_1), (e_1,0), (e_1,2e_1)\} . 
\end{eqnarray*}
Then  $|V(\sigma * \tau)|=m+n+3$. 
Let
\begin{eqnarray*}
f(e_1) &:=& \prod_{i=2}^DP_{\bf p}(e_1, 2e_1+e_i \not \in \omega(e_1+e_i)) P_{\bf p}(e_1, 2e_1-e_i \not \in \omega(e_1-e_i)) \\ 
f(0) &:=& \prod_{i=2}^DP_{\bf p}(0 \not \in \omega(e_i)) P_{\bf p}(0 \not \in \omega(-e_i))\\
f(-e_1) &:=& \prod_{i=2}^DP_{\bf p}(-e_1, -2e_1+e_i \not \in \omega(-e_1+e_i)) P_{\bf p}(-e_1, -2e_1-e_i \not \in \omega(-e_1-e_i)). 
\end{eqnarray*}
These are the probabilities such that the vertices $\{ -e_1 \pm e_i, 0\pm e_i, e_1\pm e_i\}$ ($i=2,\dots ,D$) are not weakly connected to vertices in $\{0,e_1,-e_1\} \cup \{x\in \mathbb{Z}^D \mid x_i = 2 \text{ or } -2 \}$. 
Then we have 
\begin{eqnarray*}
&& \frac{P_{\bf p}(C(0)=\sigma*\tau)}{P_{\bf p}(\omega(-e_1)=\{-2e_1,0\})P_{\bf p}(\omega(0)=\{-e_1,e_1\})P_{\bf p}(\omega(e_1)=\{0,2e_1\}) f(e_1)f(0)f(-e_1)} \\ 
&\ge & \frac{P_{\bf p}(C(-2e_1)=\sigma) P_{\bf p}(C(2e_1)=\tau)}{P_{\bf p}(-2e_1 \not\in  \omega(-e_1)) P_{\bf p}(2e_1 \not\in \omega(e_1))}  . 
\end{eqnarray*}
We can calculate 
\begin{eqnarray*}
&& f(e_1) = f(-e_1) = \left( \sum_{i=0}^{d-2} p_i \frac{(d-i)(d-i-1)}{d(d-1)}\right)^{d-2} > 0, \ \ 
f(0) = \left( \sum_{i=0}^{d-1} p_i \frac{d-i}{d}\right)^{d-2} > 0,  \\
&& P_{\bf p}(\omega(-e_1)=\{-2e_1,0\}) = P_{\bf p}(\omega(0)=\{-e_1,e_1\}) = P_{\bf p}(\omega(e_1)=\{0,2e_1\}) = \frac{2p_2}{d(d-1)}\\ 
&& P_{\bf p}(-2e_1 \not\in  \omega(-e_1)) = P_{\bf p}(2e_1 \not\in \omega(e_1))=\sum_{i=0}^{d-1} p_i \frac{d-i}{d}  . 
\end{eqnarray*}
Hence 
\begin{eqnarray*}
P_{\bf p}(C(0)=\sigma*\tau) 
&\ge & 
a(d,{\bf p}) P_{\bf p}(C(-2e_1)=\sigma) P_{\bf p}(C(2e_1)=\tau). 
\end{eqnarray*}
Thus 
\begin{eqnarray*}
&& a(d,{\bf p}) \frac{\pi_{m}}{m}\frac{\pi_{n}}{n} \\
&=& a(d,{\bf p})  \left( \sum_{\sigma , |\sigma|=m, \tr(\sigma)=-2e_1} P_{\bf p}(C(-2e_1)=\sigma) \right) \left( \sum_{\tau, |\tau|=n, \bl(\tau)=2e_1} P_{\bf p}(C(2e_1)=\tau) \right) \\ 	
&=& \sum_{\sigma , |\sigma|=m, \tr(\sigma)=-2e_1}  \sum_{\tau, |\tau|=n, \bl(\tau)=2e_1} a(d,{\bf p})  P_{\bf p}(C(-2e_1)=\sigma) P_{\bf p}(C(2e_1)=\tau) \\ 
&\le & \sum_{\sigma , |\sigma|=m, \tr(\sigma)=-2e_1}  \sum_{\tau, |\tau|=n, \bl(\tau)=2e_1} P_{\bf p}(C(0)=\sigma*\tau)  \\ 
&\le& \sum_{\rho , |\rho|=m+n+3, \bl(\rho)=0}  P_{\bf p}(C=\rho)  \\ 
&=& \frac{\pi_{m+n+3}}{m+n+3}. 
\end{eqnarray*}
\end{proof}
From this lemma, we get 
\begin{eqnarray*}
a(d,{\bf p}) \frac{\pi_{m+n+3}}{m+n+3} \ge \left( a(d,{\bf p}) \frac{\pi_{m}}{m}\right) \left(a(d,{\bf p}) \frac{\pi_{n}}{n} \right).  
\end{eqnarray*}
Thus $b_n:= - \log a(d,{\bf p}) - \log \pi_n + \log n > 0$ satisfies  
$b_{m+n+3} \le b_m+b_n $
for $n,m\ge 2$. 
As Fekete's subadditive lemma, we can prove 
\begin{lemma}
\begin{eqnarray*}
\lim_{n\to \infty}\frac{b_n}{n} = \inf_{n \in \mathbb{N}}\frac{b_n}{n}  . 
\end{eqnarray*}
\end{lemma}

From this lemma, there is
\begin{eqnarray*}
\lim_{n\to \infty}\frac{b_n}{n} 
= \lim_{n\to \infty}\frac{- \log a(d,{\bf p}) - \log \pi_n + \log n}{n}
= \lim_{n\to \infty}\frac{- \log \pi_n}{n} 
= \zeta(\mathbb{L}^D,{\bf p}). 
\end{eqnarray*}
Moreover since
$\zeta(\mathbb{L}^D,{\bf p}) = \inf_{n \in \mathbb{N}}\frac{b_n}{n} = \inf_{n \in \mathbb{N}} \frac{-1}{n} \log \left( a(d,{\bf p})\frac{\pi_n}{n} \right), $
we have 
\begin{eqnarray*}
P_{\bf p}(|C|=n)=\pi_n\le \frac{n}{a(d,{\bf p})}e^{-n \zeta(\mathbb{L}^D,{\bf p})}
\end{eqnarray*}
for any $n \in \mathbb{N}$. 

If $\chi(\mathbb{L}^D,{\bf p}) <\infty$, then by the Theorem \ref{thm:NumVertex}, we obtain 
\begin{eqnarray*}
\zeta(\mathbb{L}^D,{\bf p}) 
= \lim_{n\to \infty}\frac{- \log P_{\bf p}(|C|=n)}{n} 
\ge \lim_{n\to \infty}\frac{- \log P_{\bf p}(|C|\ge n )}{n} 
\ge  \frac{c({\bf p},d)}{2\chi(\mathbb{L}^D,{\bf p})^2}  
> 0. 
\end{eqnarray*}
\end{proof}

\begin{theorem}
\label{thm:stronglog}
Suppose $D\ge 2$, and if $p_i>0$ and $i\ge 1$ then $p_j>0$  for all $j \ge i$. 
Then there exists the limit 
\begin{eqnarray*}
\tilde{\zeta}(\mathbb{L}^D,{\bf p}):= \lim_{n\to \infty} \left\{ \frac{-\log P_{\bf p}(|\tilde{C}|=n)}{n} \right\} \ge 0
\end{eqnarray*}
and if $\tilde{\chi}(\mathbb{L}^D,{\bf p}) <\infty$, then $\tilde{\zeta}(\mathbb{L}^D,{\bf p}) >0$. 
Moreover,  
\begin{eqnarray*}
P_{\bf p}(|\tilde{C}|=n) \le \frac{n}{\tilde{a}(d,{\bf p})} e^{-n\tilde{\zeta}(\mathbb{L}^D,{\bf p})}
\end{eqnarray*}
for any $n\in \mathbb{N}$, where 
\begin{eqnarray*}
\tilde{a}(d,{\bf p}) := 
\min \left\{ 1, \frac{p_{k+1}}{p_k} \frac{k+1}{d-k}, \frac{p_{k+2}}{p_{k+1}} \frac{k+2}{d-k-1}, \dots, \frac{p_d}{p_{d-1}} d\right\}^2
\end{eqnarray*}
for $k:=\min\{j>0 \mid p_j>0\}$. 
\end{theorem}

\begin{proof}
The proof is also based on \cite{MR0496290}. 
Let $\pi_n:=P_{\bf p}(|\tilde{C}|=n)$. 
If we prove the following inequality, then we obtain the conclusion by the same proof of Theorem \ref{thm:weaklog}. 
\begin{eqnarray*}
\frac{\pi_{m+n}}{m+n} \ge \tilde{a}(d,{\bf p}) \frac{\pi_{m}}{m}\frac{\pi_{n}}{n} 
\end{eqnarray*}
for all $n,m \ge 1$. 
As the proof of Lemma \ref{lem:Fekete3}, we define $\tr(S)$ and $\bl(S)$ for a subgraph $S$ of $\mathbb{L}^D$, and standard basis $e_i\in \mathbb{R}^D$.  
Let $\sigma , \tau$ be connected subgraphs of $\mathbb{L}^D$ with $\tr(\sigma)=0$, $\bl(\tau)=e_1$ and $|\sigma|=m$, $|\tau|=n$ for $m,n \in \mathbb{N}$.  
We define the subgraph $\sigma * \tau$ such that 
\begin{eqnarray*}
V(\sigma * \tau) &:=& V(\sigma) \cup V(\tau) \\ 
E(\sigma * \tau) &:=& E(\sigma) \cup E(\tau ) \cup \{ (0, e_1), (e_1, 0) \} . 
\end{eqnarray*}
Then  $|V(\sigma * \tau)|=m+n$. 
Let $\sigma (0):=\{v \in L_0 \mid (0,v) \in E(\sigma)\}$, $\sigma' (0) := \sigma (0) \cup \{e_1\} $, and $\tau(e_1):=\{v \in L_{e_1} \mid (e_1,v) \in E(\tau)\}$, $\tau'(e_1):= \tau(e_1) \cup \{0\}$. 
Then we have 
\begin{eqnarray*}
&& P_{\bf p}(\tilde{C}(0)=\sigma * \tau )  \\ 
&=& P_{\bf p}(\tilde{C}(0)=\sigma) \frac{P_{\bf p}(\omega(0)=\sigma'(0))}{P_{\bf p}(\omega(0)=\sigma(0))}
P_{\bf p}(\tilde{C}(e_1)=\tau ) \frac{P_{\bf p}(\omega(e_1)=\tau'(e_1))}{P_{\bf p}(\omega(e_1)=\tau(e_1))} \\
&\ge & P_{\bf p}(\tilde{C}(0)=\sigma) P_{\bf p}(\tilde{C}(e_1)=\tau )
\min \left\{ 1, \frac{p_{k+1}}{p_k} \frac{k+1}{d-k}, \frac{p_{k+2}}{p_{k+1}} \frac{k+2}{d-k-1}, \dots, \frac{p_d}{p_{d-1}} d\right\}^2\\ 
&=& \tilde{a}(d, {\bf p})P_{\bf p}(\tilde{C}(0)=\sigma) P_{\bf p}(\tilde{C}(e_1)=\tau )
\end{eqnarray*}
for $k=\min\{j>0 \mid p_j>0\}$.
Thus 
\begin{eqnarray*}
&& \tilde{a}(d,{\bf p}) \frac{\pi_{m}}{m}\frac{\pi_{n}}{n} \\
&=& \tilde{a}(d,{\bf p})  \left( \sum_{\sigma , |\sigma|=m, \tr(\sigma)=0} P_{\bf p}(\tilde{C}(0)=\sigma) \right) 
\left( \sum_{\tau, |\tau|=n, \bl(\tau)=e_1} P_{\bf p}(\tilde{C}(e_1)=\tau) \right) \\ 
&=& \sum_{\sigma , |\sigma|=m, \tr(\sigma)=0}  \sum_{\tau, |\tau|=n, \bl(\tau)=e_1} \tilde{a}(d,{\bf p})  P_{\bf p}(\tilde{C}(0)=\sigma) P_{\bf p}(\tilde{C}(e_1)=\tau) \\ 
&\le & \sum_{\sigma , |\sigma|=m, \tr(\sigma)=0}  \sum_{\tau, |\tau|=n, \bl(\tau)=e_1} P_{\bf p}(\tilde{C}(0)=\sigma*\tau)  \\ 
&\le& \sum_{\rho , |\rho|=m+n, \bl(\rho)=0}  P_{\bf p}(\tilde{C}=\rho)  \\ 
&=& \frac{\pi_{m+n}}{m+n}. 
\end{eqnarray*}
\end{proof}

This proof is also valid for weak cluster.

\vspace{5mm}
\noindent 
Mamoru Tanaka,\\ 
Advanced Institute for Materials Research, Tohoku University, Sendai, 980-8577 Japan\\ 
E-mail: mamoru.tanaka@wpi-aimr.tohoku.ac.jp

\end{document}